\documentclass[11pt, twoside, leqno]{article}
\usepackage{amssymb}
\usepackage{amsmath}
\usepackage{amsthm}
\usepackage{color}
\usepackage{mathrsfs}
\usepackage{graphicx}
\usepackage{indentfirst}
\usepackage{txfonts}

\allowdisplaybreaks

\def\XXint#1#2#3{{\setbox0=\hbox{$#1{#2#3}{\int}$ }
\vcenter{\hbox{$#2#3$ }}\kern-.6\wd0}}

\def\({\left(}
\def \){ \right)}

\newtheorem{theorem}{Theorem}[section]
\newtheorem{lemma}[theorem]{Lemma}

\newtheorem{proposition}[theorem]{Proposition}

\theoremstyle{definition}
\newtheorem{remark}[theorem]{Remark}

\renewcommand{\appendix}{\par
   \setcounter{section}{0}%
   \setcounter{subsection}{0}%
   \setcounter{subsubsection}{0}%
   \gdef\thesection{\@Alph\c@section}%
   \gdef\thesubsection{\@Alph\c@section.\@arabic\c@subsection}%
   \gdef\theHsection{\@Alph\c@section.}%
   \gdef\theHsubsection{\@Alph\c@section.\@arabic\c@subsection}%
   \csname appendixmore\endcsname
 }

\numberwithin{equation}{section}

\begin{document}

\arraycolsep=1pt

\title{\bf\Large $L_x^p\rightarrow L^q_{x,u}$ estimates for dilated averages over planar curves
\footnotetext{\hspace{-0.35cm} 2020 {\it
Mathematics Subject Classification}. Primary 42B10;
Secondary 42B15.
\endgraf {\it Key words and phrases.} generalised Radon transforms,  $L_x^p-L^q_{x,u}$ estimates, local smoothing estimates, Fourier integral operators.
\endgraf Junfeng Li was supported by Natural Science Foundation of China (No.~12071052) and the Fundamental Research Funds for the Central Universities. Naijia Liu was supported by China Postdoctoral Science Foundation (No.~2022M723673). Zengjian Lou was supported by Natural Science Foundation of China (No.~12071272). Haixia Yu was supported by Natural Science Foundation of China (No.~12201378), Guangdong Basic and Applied Basic Research Foundation (No.~2023A1515010635) and STU Scientific Research Foundation for Talents (No.~NTF21038).}}
\author{Junfeng Li, Naijia Liu, Zengjian Lou and Haixia Yu\footnote{Corresponding author.}}

\date{}

\maketitle

\vspace{-0.7cm}

\begin{abstract}
 In this paper, we consider the $L_x^p(\mathbb{R}^2)\rightarrow L_{x,u}^q(\mathbb{R}^2\times [1,2])$ estimate for the operator $T$ along a dilated plane curve $(ut,u\gamma(t))$, where
  $$Tf(x,u):=\int_{0}^{1}f(x_1-ut,x_2-u \gamma(t))\,\textrm{d}t,$$
 $x:=(x_1,x_2)$ and $\gamma$ is a general plane curve satisfying some suitable smoothness and curvature conditions. We show that $T$ is $L_x^p(\mathbb{R}^2)$ to $L_{x,u}^q(\mathbb{R}^2\times [1,2])$ bounded whenever $(\frac{1}{p},\frac{1}{q})\in \square \cup \{(0,0)\}\cup \{(\frac{2}{3},\frac{1}{3})\}$ and $1+(1 +\omega)(\frac{1}{q}-\frac{1}{p})>0$, where the trapezium $\square:=\{(\frac{1}{p},\frac{1}{q}):\ \frac{2}{p}-1\leq\frac{1}{q}\leq \frac{1}{p}, \frac{1}{q}>\frac{1}{3p}, \frac{1}{q}>\frac{1}{p}-\frac{1}{3}\}$ and $\omega:=\limsup_{t\rightarrow 0^{+}}\frac{\ln|\gamma(t)|}{\ln t}$. This result is sharp except for some borderline cases. On the other hand, in a smaller $(\frac{1}{p},\frac{1}{q})$ region, we also obtain the almost sharp estimate $T : L_x^p(\mathbb{R}^2)\rightarrow L_{x}^q(\mathbb{R}^2)$ uniformly for $u\in [1,2]$. These results imply that the operator $T$ has the so called local smoothing phenomenon, i.e., the $L^q$ integral about $u$ on $[1,2]$ extends the region of $(\frac{1}{p},\frac{1}{q})$ in uniform estimate $T : L_x^p(\mathbb{R}^2)\rightarrow L_{x}^q(\mathbb{R}^2)$.

\end{abstract}

\section{Introduction}

It is well known that the theory of maximal functions plays a central role in harmonic analysis and partial
differential equations. Which has attracted a lot of attention in the last decades. Our topic in this paper originated in the estimate for maximal function. Let us begin with a definition of the \emph{maximal function} $M_t$ along the straight line $(t,t)$,
\begin{align*}
M_tf(x):=\sup_{u\in (0,\infty)}\frac{1}{u}\left|\int_{0}^{u}f(x_1-t,x_2-t)\,\textrm{d}t\right|
\end{align*}
for appropriate functions $f$ on $\mathbb{R}^2$. Here and hereafter, $x:=(x_1,x_2)$. From Stein \cite{Stein}, it follows that $M_t$ is bounded on $L^p(\mathbb{R}^2)$ for all $p\in(1,\infty]$ because $M_t$ is essentially the Hardy-Littlewood maximal function. By a straightforward change of
variables, we can rewrite
\begin{align*}
M_tf(x)=\sup_{u\in (0,\infty)}\left|\int_{0}^{1}f(x_1-ut,x_2-u t)\,\textrm{d}t\right|.
\end{align*}
It is natural to replace the straight line $(t,t)$ in $M_t$ by a general plane curve $(t,\gamma(t))$, such as the parabola $(t,t^2)$, where $\gamma$ satisfies some suitable smoothness and curvature conditions. We now consider the \emph{maximal function} $M_\gamma$ along a curve $(t,\gamma(t))$,
\begin{align*}
M_\gamma f(x):=\sup_{u\in (0,\infty)}\left|\int_{0}^{1}f(x_1-ut,x_2-u \gamma(t))\,\textrm{d}t\right|.
\end{align*}

For notational convenience, we define the operator $T$ along a dilated plane curve $(ut,u\gamma(t))$ as
\begin{align}\label{eq:1.0}
Tf(x,u):=\int_{0}^{1}f(x_1-ut,x_2-u \gamma(t))\,\textrm{d}t,
\end{align}
then
\begin{align*}
M_\gamma f(x)=\sup_{u\in (0,\infty)}\left|Tf(x,u)\right|.
\end{align*}
It is natural to ask whether the $L_x^p(\mathbb{R}^2)\rightarrow L_{x}^q(\mathbb{R}^2)$ estimate for $M_\gamma$ is valid or not. However, by Liu and Yu \cite[Remark 1.3]{LiuYu}, the maximal function $M_\gamma$ does not map $L_x^p(\mathbb{R}^2)$ to $L_{x}^q(\mathbb{R}^2)$ (even when $\sup_{u\in (0,\infty)}$ in $M_\gamma$ is replaced by $\sup_{u\in (0,1)}$), unless $p=q$. Related works can be found in \cite{Bour86, SSte, IKM, KLO} and references therein. However, if we take supremum over $u\in [1,2]$ in $M_\gamma$, i.e.,
\begin{align*}
\bar{M}_\gamma f(x):=\sup_{u\in [1,2]}\left|Tf(x,u)\right|.
\end{align*}
Then $\bar{M}_\gamma$ is actually bounded from $L_x^p(\mathbb{R}^2)$ to $L_{x}^q(\mathbb{R}^2)$ for some $q>p$, see \cite{LiuYu}. Indeed, this result can be rewritten as
\begin{align}\label{eq:1.1}
\left\|Tf\right\|_{L_x^{q}L_u^{\infty}(\mathbb{R}^{2}\times [1,2])} \lesssim \|f\|_{L^{p}(\mathbb{R}^{2})}.
\end{align}
This phenomenon is called $L^p(\mathbb{R}^2)$-improving. We refer the reader to \cite{TW, Sch, LeeL, LWZ} for the details of this phenomenon.

Our goal in this paper is to establish the $L_x^p(\mathbb{R}^2)\rightarrow L_{x,u}^q(\mathbb{R}^2\times [1,2])$ estimate for $T$. By $\textrm{H}\ddot{\textrm{o}}\textrm{lder}$'s inequality, it is easy to see that, when $(\frac{1}{p}, \frac{1}{q})$ verifies the inequality \eqref{eq:1.1}, we have the $L_x^p(\mathbb{R}^2)\rightarrow L_{x,u}^q(\mathbb{R}^2\times [1,2])$ estimate for $T$. However, the $(\frac{1}{p}, \frac{1}{q})$ region for $T:\ L_x^p(\mathbb{R}^2)\rightarrow L_{x,u}^q(\mathbb{R}^2\times [1,2])$ becomes large. We now state our main results.

\begin{theorem}\label{thm1}
Assume $\gamma\in C^{N}(0,1]$ with $N\in\mathbb{N}$ large enough, $\lim_{t\rightarrow 0^+}\gamma(t)=0$ and $\gamma$ is monotonic on $(0,1]$. Moreover, $\gamma$ satisfies the following two conditions:
\begin{enumerate}\label{curve gamma}
  \item[\rm(i)] there exist positive constants $\{C^{(j)}_{1}\}_{j=1}^{2}$ such that $|\frac{t^{j}\gamma^{(j)}(t)}{\gamma(t)}|\geq C^{(j)}_{1}$  for any $t\in (0,1]$;
  \item[\rm(ii)] there exist positive constants $\{C^{(j)}_{2}\}_{j=1}^{N}$ such that $|\frac{t^{j}\gamma^{(j)}(t)}{\gamma(t)}|\leq C^{(j)}_{2}$ for any $t\in (0,1]$.
\end{enumerate}
Then, there exists a positive constant $C$ such that for all $f\in L_x^{p}(\mathbb{R}^{2})$,
\begin{align*}
\left\|Tf\right\|_{L_{x,u}^q(\mathbb{R}^2\times [1,2])} \leq C \|f\|_{L_x^{p}(\mathbb{R}^{2})},
\end{align*}
if $(\frac{1}{p},\frac{1}{q})\in \square \cup \{(0,0)\}\cup \{(\frac{2}{3},\frac{1}{3})\}$ and $1+(1 +\omega)(\frac{1}{q}-\frac{1}{p})>0$. Here and hereafter, the trapezium $\square:=\{(\frac{1}{p},\frac{1}{q}):\ \frac{2}{p}-1\leq\frac{1}{q}\leq \frac{1}{p}, \frac{1}{q}>\frac{1}{3p}, \frac{1}{q}>\frac{1}{p}-\frac{1}{3}\}$ and $\omega:=\limsup_{t\rightarrow 0^{+}}\frac{\ln|\gamma(t)|}{\ln t}$.
\end{theorem}

We show the conditions of $(\frac{1}{p}, \frac{1}{q})$ in Theorem \ref{thm1} are necessary, which means that Theorem \ref{thm1} is sharp except for some borderline cases. Moreover, we can also find that the region of $(\frac{1}{p}, \frac{1}{q})$ in Theorem \ref{thm1} remains unchanged when $\omega\in (0,2)$, but decreases as $\omega$ increases when $\omega\in [2,\infty)$. Since the curvature of a curve is key in the proofs, which may provide an underlying cause of this phenomenon. At the same time, for the infinitely flat curve\footnote{A curve $\gamma$ is called a infinitely flat curve if it satisfies $\gamma^{(k)}(0)=0$ for any $k\in \mathbb{N}_0$, such as $\gamma(t):=e^{-1/t}$ or $e^{-1/t^2}$. These curves become very close to being a ``line"  at the origin.}, it is reasonable to conjecture that the estimate $T:\ L_x^p(\mathbb{R}^2)\rightarrow L_{x,u}^q(\mathbb{R}^2\times [1,2])$ holds only for $p=q\in [1,\infty]$.

\begin{theorem}\label{thm2}
Let $\gamma$ be defined as in Theorem \ref{thm1}. Then, there exists a positive constant $C$ such that for all $f\in L_x^{p}(\mathbb{R}^{2})$, the estimate
\begin{align*}
\left\|Tf\right\|_{L_{x,u}^q(\mathbb{R}^2\times [1,2])} \leq C \|f\|_{L_x^{p}(\mathbb{R}^{2})},
\end{align*}
holds, only if the following conditions are satisfied:
\begin{enumerate}\label{necessary}
  \item[\rm(i)] $(\frac{1}{p},\frac{1}{q})$ satisfy $1+(1 +\omega)(\frac{1}{q}-\frac{1}{p})\geq0$;
  \item[\rm(ii)] $(\frac{1}{p},\frac{1}{q})$ satisfy $\frac{2}{p}-1\leq\frac{1}{q}\leq \frac{1}{p}$;
  \item[\rm(iii)] $(\frac{1}{p},\frac{1}{q})$ satisfy $\frac{1}{q}\geq\frac{1}{3p}$;
  \item[\rm(iv)] $(\frac{1}{p},\frac{1}{q})$ satisfy $\frac{1}{q}\geq\frac{1}{p}-\frac{1}{3}$.
\end{enumerate}
 \end{theorem}

For these results, we would like to give the following some remarks.

\begin{remark}\label{remark 1}
Here are some examples of curves satisfying the conditions (i) and (ii) of Theorem \ref{thm1}. We may add a characteristic function $\chi_{(0,\varepsilon_0]}(t)$ to these curves if necessary, where $\varepsilon_0$ is small enough.
\begin{enumerate}
\item[\rm(1)] $\gamma_1(t):=t^d$, where $d\in(0,\infty)$ and $d\neq1$;
\item[\rm(2)] $\gamma_2(t):=t^{d-1}\ln(1+t)$, where $d\in(1,\infty)$;
\item[\rm(3)] $\gamma_3(t):=a_dt^d+a_{d+1}t^{d+1}+\cdots+a_{d+m}t^{d+m}$, where $d\geq 2$, $d\in\mathbb{N}$ and $m\in\mathbb{N}_0$ , $a_d\neq 0$, i.e., $\gamma_3$ is a polynomial of degree at least $d$ with no linear term and constant term;
\item[\rm(4)] $\gamma_4(t):=\sum_{i=1}^{d}\beta_i t^{\alpha_i}$, where $\alpha_i\in(0,\infty)$ for all $i=1,2, \cdots, d$, $\min_{i\in\{1,2, \cdots, d\}}\{\alpha_i\}_{i=1}^{d}\neq 1$ and $d\in \mathbb{N}$;
\item[\rm(5)] $\gamma_5(t):=1-\sqrt{1-t^2}$, or $t\sin t$, or $t-\sin t$, or $1-\cos t$, or $e^t-t-1$;
\item[\rm(6)] $\gamma_6(t)$ is a smooth function on $[0,1]$ satisfying $\gamma(0)=\gamma'(0)=\cdots=\gamma^{(d-1)}(0)=0$ and $\gamma^{(d)}(0)\neq0$, where $d\geq 2$ and $d\in\mathbb{N}$. Note that $\gamma_6$ is finite type $d$ at the origin (see, Iosevich \cite{Iose}). $\gamma_3$ and $\gamma_5$ provide some special cases of $\gamma_6$.
\end{enumerate}
\end{remark}

\begin{remark}\label{remark 2}
If $\gamma(t):=\gamma_1(t)=t^d$ with $d\in(0,\infty)$ and $d\neq1$, it follows that $1+(1 +\omega)(\frac{1}{q}-\frac{1}{p})>0$ is equivalent to
$(\frac{1}{p},\frac{1}{q})\in\{(\frac{1}{p},\frac{1}{q}):\ \frac{1}{q}>\frac{1}{p}- \frac{1}{d+1}\}$. As a result, we can express the region of $(\frac{1}{p}, \frac{1}{q})$ in Theorem \ref{thm1} as in Figure \ref{Figure:1}. On the other hand, by a simple calculation, $\gamma_2$, $\gamma_3$ and $\gamma_6$ share the same region with $\gamma_1$.
\begin{figure}[htbp]
  \centering
  \includegraphics[width=3.8in]{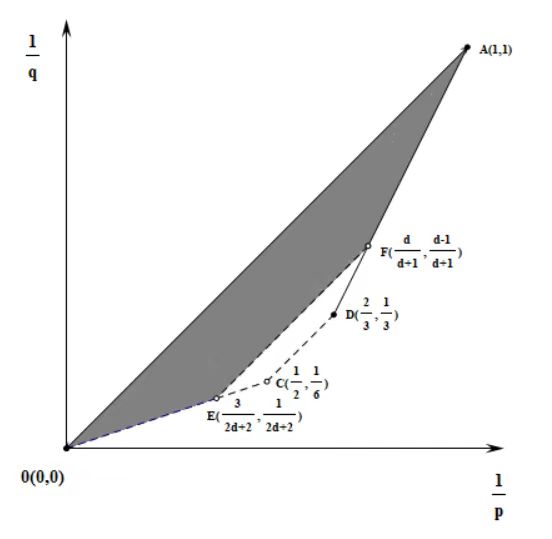}\\
  \caption{The region of $(\frac{1}{p}, \frac{1}{q})$ in Theorem \ref{thm1} for $\gamma=\gamma_1$, $\gamma_2$, $\gamma_3$ or $\gamma_6$.}\label{Figure:1}
\end{figure}
\end{remark}

In \cite{H}, Hickman obtained the almost sharp (up to an endpoint) $L_x^p(\mathbb{R}^n)\rightarrow L_{x,u}^q(\mathbb{R}^n\times [1,2])$ estimate for the $n$-dimensional version of $T$ along the moment curve $(t,t^2,\cdots,t^n)$ by establishing a weak-type endpoint estimate. Our main results in Theorems \ref{thm1} and \ref{thm2} are the two-dimensional case, but with a general plane curve including the parabola $\gamma(t):=t^2$. However, it is worth pointing out that our proofs are completely different. In particular, we establish a fixed time estimate and a local smoothing estimate for a general Fourier integral operator, based on the works of Seeger, Sogge and Stein \cite{SSS}, Stein \cite[Chapter IX]{Stein} and, Gao, Liu, Miao and Xi \cite{GLMX}. We remark that \cite{GLMX} can be viewed as a generalization of the local smoothing conjecture in $2+1$ dimensions obtained in Guth, Wang and Zhang \cite{GWZ}. Furthermore, as stated above, the region of $(\frac{1}{p}, \frac{1}{q})$ in estimate $T:\ L_x^p(\mathbb{R}^2)\rightarrow L_{x,u}^q(\mathbb{R}^2\times [1,2])$ depends on the value $\omega=\limsup_{t\rightarrow 0^{+}}\frac{\ln|\gamma(t)|}{\ln t}$, this is a new situation, to the best of our knowledge. At the same time, we also show that $T$ has the local smoothing phenomenon by obtaining the almost sharp estimate $T : L_x^p(\mathbb{R}^2)\rightarrow L_{x}^q(\mathbb{R}^2)$ uniformly for $u\in [1,2]$ in Theorem \ref{thm3}.

We should point out some historical backgrounds about the study of the operator $T$. When $\gamma(t):=t^2$, based on the method of our proofs, we can not show the $L_x^2(\mathbb{R}^2)\rightarrow L_{x,u}^6(\mathbb{R}^2\times [1,2])$ boundedness of $T$. However, the estimate $T:\ L_x^2(\mathbb{R}^2)\rightarrow L_{x,u}^6(\mathbb{R}^2\times [1,2])$ is already known to hold essentially by Strichartz \cite{Str}, Schlag and Sogge \cite{SchS}. We also remark that this estimate for dilated averages over circles is equivalent to the famous Stein-Tomas Fourier restriction theorem for a conic surface. On the other hand, for other rangers of $(\frac{1}{p}, \frac{1}{q})$, the estimates $T:\ L_x^p(\mathbb{R}^2)\rightarrow L_{x,u}^q(\mathbb{R}^2\times [1,2])$ about $(t,t^2)$ obtained in Gressman \cite{Gr1, Gr2} by using the combinatorial argument, which can be traced back to Christ \cite{Chri}.

The motivation for our curved variant also comes from the \emph{maximal function} $\mathcal{M}$ defined by taking averages over curve $(t,\gamma(t))$,
\begin{align*}
\mathcal{M}f(x):=\sup_{\epsilon\in(0,\infty)}
\frac{1}{2\epsilon}\int_{-\epsilon}^{\epsilon}\left|f(x_1-t,x_2-\gamma(t))\right|
\,\textrm{d}t,
\end{align*}
which is also been extensively studied. For $\gamma(t):=t^2$, Nagel, Riviere and Wainger \cite{NRW76} established the $L_x^p(\mathbb{R}^2)$ boundedness of $\mathcal{M}$ for any $p\in(1,\infty]$. Stein \cite{Ste1} obtained this boundedness for homogeneous curves. Stein and Wainger \cite{SW1} studied some smooth curves. Later it was extended to more general families of curves; see, for example, \cite{SW,CCVWW,CVWW}.

Let us now give an overview of the proofs of Theorems \ref{thm1} and \ref{thm2}. By a non-isotropic dilation, the proof is reduced to an uniform estimate for the averages operator $\widetilde{T}_{j}$ along a dilated plane curve $(ut, u\Gamma_{j}(t))$, where $\Gamma_{j}(t):=\frac{\gamma(2^jt)}{\gamma(2^j)}$ behaves uniformly in the parameter $j$.
Specifically, the second derivative of $\Gamma_{j}(t)$ is bounded below uniformly in $j$. This non-isotropic dilation seems necessary, which yields the slightly different
notation $\omega=\limsup_{t\rightarrow 0^{+}}\frac{\ln|\gamma(t)|}{\ln t}$. For the estimate of $\widetilde{T}_{j}$, by interpolation, it suffices to obtain an endpoint estimate at $D(\frac{2}{3},\frac{1}{3})$ and an estimate at the half open line $(C, M]$ in Figure \ref{Figure:3}. For the
$D(\frac{2}{3},\frac{1}{3})$ case, the corresponding estimate can be established by Lemmas \ref{lemma 2.1} and \ref{lemma 2.2}. For the half open line $(C, M]$ case, by using the theory of oscillatory integrals and stationary phase estimates, it is enough to obtain a local smoothing estimate for
\begin{align*}
F_{j,k}f(x,u):=\int_{\mathbb{R}^{2}} e^{i\Psi_j(x,u,\xi)}a_{j,k}(x,u,\xi) \hat{f}(\xi)\,\textrm{d}\xi,
\end{align*}
which is defined in \eqref{eq:3.5}. The operator $F_{j,k}$ is related to a class of Fourier integral operators studied in many papers (see, for instance, \cite{Ho71, DH72, Sog, B}). These Fourier integral operators are originated from the study of pseudo-differential operators or half-wave propagator. Furthermore, $F_{j,k}$ is also essentially a Fourier integral operator with phase function $x\cdot\xi-u|\xi|$, and the local smoothing estimate for $F_{j,k}$ will be reduced to a decoupling inequality for cones due to Bourgain and Demeter \cite{BoD}. Here, for a general $n$ dimension Fourier integral operator in \eqref{eq:3.a} satisfying the cinematic curvature condition, we establish a fixed time estimate about $P_kf$ in Proposition \ref{prop 3.1}. When $n=2$, we also obtain a corresponding local smoothing estimate in Proposition \ref{prop 3.2} based on the work of Gao, Liu, Miao and Xi \cite{GLMX}. Our local smoothing estimates for $F_{j,k}$ rely on Propositions \ref{prop 3.1} and \ref{prop 3.2}. It is worth noticing that to verify the so called cinematic curvature condition we need some complicated calculations. In order to obtain the necessity of the region of $(\frac{1}{p}, \frac{1}{q})$, we construct four examples. In particular, the proofs of (i) and (iv) of Theorem \ref{thm2} are a bit tricky.

This paper is structured as follows:
\begin{enumerate}
 \item[$\diamond$] In Section 2, we show the almost sharp estimate $T : L_x^p(\mathbb{R}^2)\rightarrow L_{x}^q(\mathbb{R}^2)$ uniformly for $u\in [1,2]$ in a smaller $(\frac{1}{p},\frac{1}{q})$ region comparing to that in Theorem \ref{thm1}. Furthermore, by the $L^q$ integral about $u$ on $[1,2]$, Theorem \ref{thm3} implies an important estimate at endpoint $D(\frac{2}{3},\frac{1}{3})$. It also implies that the operator $T$ has the so called local smoothing phenomenon.
 \item[$\diamond$] In Section 3, we give a proof of Theorem \ref{thm1}. An estimate at the half open line $(C, M]$ in Figure \ref{Figure:3} plays a key role, whose proof relies  heavily on the local smoothing estimates.
 \item[$\diamond$] In Section 4, we consider the necessary conditions for Theorem \ref{thm1}, i.e., Theorem \ref{thm2}. We will see that the estimate $T : L_x^p(\mathbb{R}^2)\rightarrow L_{x,u}^q(\mathbb{R}^2\times [1,2])$ in Theorem \ref{thm1} is almost sharp except for some endpoints.
\end{enumerate}

Finally, we make some convention on the notations. Throughout this paper, the letter ``$C$"  will denote a \emph{positive constant}, independent of the essential variables, but whose value may vary from line to line. Moreover, we use $C_{(u)}$ to denote a \emph{positive constant} depending on the indicated parameter $u$. The notation $a\lesssim b$ (or $a\gtrsim b$) will used to indicate the existence of a finite positive constant $C$ such that $a\leq Cb$ (or $a\geq Cb$). $a\approx b$ means $a\lesssim b$ and $b\lesssim a$. For any $x\in\mathbb{R}^n$ and $r\in(0,\infty)$, let $B(x,r):=\{y\in\mathbb{R}^n:\ |x-y|<r\}$ and $B^{\complement}(x,r)$
be its \emph{complement} in $\mathbb{R}^n$. $\hat{f}$ and $f^{\vee}$ shall denote the \emph{Fourier transform} and the \emph{inverse Fourier transform} of $f$, respectively. For $1<q\leq\infty$, we will denote $q'$ the \emph{adjoint number} of $q$, i.e., ${1}/{q}+ {1}/{q'}=1$. Let $\mathbb{Z}_0^{-}:=\{0,\,-1,\,-2,...\}$, $\mathbb{N}:=\{1,\,2,...\}$, $\mathbb{N}_0:=\{0,\,1,\,2,...\}$ and $\mathbb{R}^{+}:=(0,\infty)$. For any set $E$, we use $\chi_E$ to denote the \emph{characteristic function} of $E$.

\section{The local smoothing phenomenon}

Recall that
\begin{align*}
Tf(x,u)=\int_{0}^{1}f(x_1-ut,x_2-u \gamma(t))\,\textrm{d}t
\end{align*}
with $u\in [1,2]$, which is an operator of convolution type. When $u=1$, $Tf(\cdot,1)$ is a classical operator in harmonic analysis. For some related results, we refer to \cite{Lit, Ob1, Ob2, Sto1, Sto2} and the references therein. In this section, our main goal is to determine the almost optimal range of exponents $(\frac{1}{p},\frac{1}{q})$ such that the $L_x^p(\mathbb{R}^2)\rightarrow L_{x}^q(\mathbb{R}^2)$ estimate for $T$ holds uniformly for $u\in [1,2]$. We now state our result.

\begin{theorem}\label{thm3}
Assume $\gamma\in C^{2}(0,1]$, $\lim_{t\rightarrow 0^+}\gamma(t)=0$ and $\gamma$ is monotonic on $(0,1]$. Moreover, $\gamma$ satisfies the following two conditions:
\begin{enumerate}\label{curve gamma}
  \item[\rm(i)] there exist positive constants $\{C^{(j)}_{1}\}_{j=1}^{2}$ such that $|\frac{t^{j}\gamma^{(j)}(t)}{\gamma(t)}|\geq C^{(j)}_{1}$  for any $t\in (0,1]$;
  \item[\rm(ii)] there exists a positive constant $C^{(1)}_{2}$ such that $|\frac{t\gamma'(t)}{\gamma(t)}|\leq C^{(1)}_{2}$ for any $t\in (0,1]$.
\end{enumerate}
Then, there exists a positive constant $C$ uniformly for $u\in [1,2]$ such that for all $f\in L^{p}(\mathbb{R}^{2})$,
\begin{align}\label{eq:02.1}
\left\|Tf(\cdot,u)\right\|_{L_{x}^q(\mathbb{R}^2)} \leq C \|f\|_{L_x^{p}(\mathbb{R}^{2})},
\end{align}
if
\begin{center}
$(\frac{1}{p},\frac{1}{q}) \in \ \vartriangle $ and $(\frac{1}{p},\frac{1}{q})$ satisfy $1+(1 +\omega)(\frac{1}{q}-\frac{1}{p})>0$.
\end{center}
On the other hand, \eqref{eq:02.1} holds only if
\begin{center}
$(\frac{1}{p},\frac{1}{q})\in \ \vartriangle $ and $(\frac{1}{p},\frac{1}{q})$ satisfy $1+(1 +\omega)(\frac{1}{q}-\frac{1}{p})\geq0$.
\end{center}
Here and hereafter, the triangle $\vartriangle:=\{(\frac{1}{p},\frac{1}{q}):\ \frac{1}{2p}\leq\frac{1}{q}\leq \frac{1}{p}, \frac{1}{q}\geq\frac{2}{p}-1\}$.
\end{theorem}

\begin{remark}\label{remark 3}
\emph{The local smoothing phenomenon}. Since the estimate in \eqref{eq:02.1} is uniform for $u\in [1,2]$. By taking a $L^q$ integral about $u$ over $[1,2]$, it leads to the $L_x^p(\mathbb{R}^2)\rightarrow L_{x,u}^q(\mathbb{R}^2\times [1,2])$ estimate for $T$. Therefore, Theorem \ref{thm1} will hold in the region of $(\frac{1}{p}, \frac{1}{q})$ in Theorem \ref{thm3}.

On the other hand, from definitions of $\square$ and $\vartriangle$, we see that the trapezium $\square$ given by the open convex hull of the points $\{O, A, D, C\}$ but with the half open lines $(O, A]$ and $(D, A]$\footnote{For the convenience of statement of Theorem \ref{thm1}, we set $\square$ does not contain the points $O$ and $D$.}, and the triangle $\vartriangle$ given by the closed convex hull of the points $\{O, A, D\}$, see the following Figure \ref{Figure:2}. As a result, we have $\vartriangle\subsetneqq \square\cup \{(0,0)\}\cup \{(\frac{2}{3},\frac{1}{3})\}$, which leads to the fact that the $L^q$ integral about $u$ on $[1,2]$ will extend the region of $(\frac{1}{p},\frac{1}{q})$ in $T:\ L_x^p(\mathbb{R}^2)\rightarrow L_{x,u}^q(\mathbb{R}^2\times [1,2])$. This phenomenon is the so called local smoothing phenomenon. It is well known that this local smoothing phenomenon is very important in harmonic analysis, such as the local smoothing estimates originated from the works of Mockenhaupt, Seeger and Sogge \cite{MSS,MSS92}.

If we consider $\gamma=\gamma_1$, $\gamma_2$, $\gamma_3$ or $\gamma_6$, we can express the region of $(\frac{1}{p}, \frac{1}{q})$ in Theorem \ref{thm3} as the shaded region in Figure \ref{Figure:2}. Then, there is an additional triangle region $OEF'$, which is produced by the local smoothing phenomenon.
\begin{figure}[htbp]
  \centering
  \includegraphics[width=3.8in]{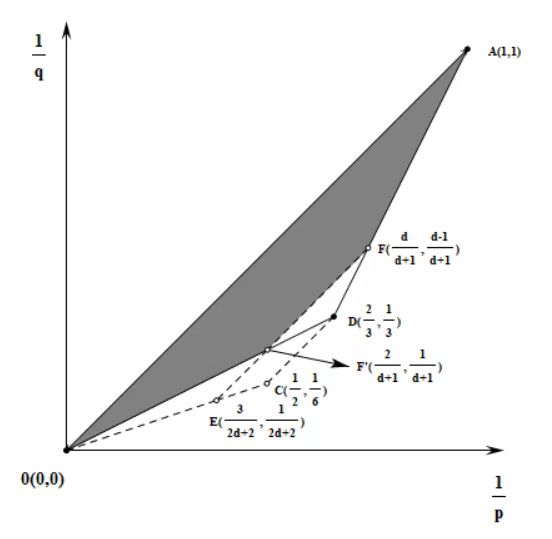}\\
  \caption{The region of $(\frac{1}{p}, \frac{1}{q})$ in Theorem \ref{thm3} for $\gamma=\gamma_1$, $\gamma_2$, $\gamma_3$ or $\gamma_6$.}\label{Figure:2}
\end{figure}
\end{remark}

The proof of Theorem \ref{thm3} relies on the following two lemmas.

\begin{lemma}\label{lemma 2.1}
(\cite[Lemma 2.3]{LiuYu}) Recall that $\omega=\limsup_{t\rightarrow 0^{+}}\frac{\ln|\gamma(t)|}{\ln t}$. Then, for all $(\frac{1}{p},\frac{1}{q})$ satisfying $\frac{1}{q}\leq \frac{1}{p}$ and $1+(1 +\omega)(\frac{1}{q}-\frac{1}{p})>0$, we have
\begin{align*}
\sum_{j\in \mathbb{Z}_0^{-}} 2^j|2^j\gamma(2^j)|^{\frac{1}{q}-\frac{1}{p}}<\infty.
\end{align*}
\end{lemma}

\begin{lemma}\label{lemma 2.2}
(\cite{Lit} or \cite[Lemma 2.1]{Cho}) Let $I$ be a compact interval and let $\phi:\ I\rightarrow\mathbb{R}$ be a $C^2$ function. Suppose there exists a positive constant $C$ such that $|\phi''(t)|\geq C$ for $t\in I$. Then $T^{\phi}_{e}$ given by
\begin{align*}
T^{\phi}_{e}f(x):=\int_{I}f(x_1-t,x_2- \phi(t))\,\textrm{d}t
\end{align*}
satisfies
 \begin{align*}
\|T^{\phi}_{e}f\|_{L^3(\mathbb{R}^2)} \leq C' \|f\|_{L^{3/2}(\mathbb{R}^{2})}
\end{align*}
with some constants $C'$ depending only on $C$.
\end{lemma}

\smallskip

We now turn to the proof of Theorem \ref{thm3}. Without loss of generality, we can assume that $f\geq 0$. We split $T$ into pieces via a standard partition of unity. Let $\psi:\ \mathbb{R}^+\rightarrow\mathbb{R}$ be a smooth function supported on $\{t\in \mathbb{R}:\ \frac{1}{2}\leq t\leq 2\}$ with $0\leq \psi(t)\leq 1$ and $\Sigma_{j\in \mathbb{Z}} \psi_j(t)=1$ for any $t> 0$, where $\psi_j(t):=\psi (2^{-j}t)$. We conclude that $T$ can be bounded by the sum of $T_j$ over $j\in \mathbb{Z}_0^{-}$, where
$$T_{j}f(x,u):=\int_{0}^{\infty}f(x_1-ut,x_2-u \gamma(t))\psi_j(t)\,\textrm{d}t,$$
which can be rewritten as, by a change of variables,
$$\int_{0}^{\infty}f(x_1-u2^jt,x_2-u \gamma(2^jt))\psi(t)2^j\,\textrm{d}t.$$
Let us set
\begin{align}\label{eq:02.02}
\widetilde{T}_{j}f(x,u):=\int_{0}^{\infty}f(x_1-ut,x_2-u \Gamma_j(t))\psi(t)\,\textrm{d}t
\end{align}
with $\Gamma_{j}(t):=\frac{\gamma(2^jt)}{\gamma(2^j)}$, and we denote the non-isotropic dilation $\delta_j$ by
$\delta_jf(x):=f(2^jx_1,\gamma(2^j)x_2)$ for any $j\in \mathbb{Z}_0^{-}$. Note that
$$|2^j\gamma(2^j)|^{\frac{1}{q}}\|\delta_jf\|_{L^{q}(\mathbb{R}^{2})}=\|f\|_{L^{q}(\mathbb{R}^{2})}~~~
 \textrm{and} ~~~ \delta_j\left(  T_{j}f\right)=2^j  \widetilde{T}_{j}(\delta_jf) $$
holds for all $j\in \mathbb{Z}_0^{-}$, it suffices to prove that
\begin{align*}
\sum_{j\in \mathbb{Z}_0^{-}} 2^j|2^j\gamma(2^j)|^{\frac{1}{q}-\frac{1}{p}} \left\| \widetilde{T}_{j} \right\|_{L_x^{p}(\mathbb{R}^{2})\rightarrow L_x^{q}(\mathbb{R}^{2})}\lesssim 1
\end{align*}
uniformly in $u\in [1,2]$. From Lemma \ref{lemma 2.1}, if $(\frac{1}{p}, \frac{1}{q})$ lies in the region of Theorem \ref{thm3}, it follows that the series $\sum_{j\in \mathbb{Z}_0^{-}} 2^j|2^j\gamma(2^j)|^{\frac{1}{q}-\frac{1}{p}}$ is bounded. Therefore, for all $(\frac{1}{p},\frac{1}{q}) \in \ \vartriangle $, we need only give an estimate $\widetilde{T}_{j} :\ L_x^p(\mathbb{R}^2)\rightarrow L_{x}^q(\mathbb{R}^2)$ uniformly in $u\in [1,2]$ and $j\in \mathbb{Z}_0^{-}$.

By Minkowski's inequality, for $p\in[1,\infty]$, it is easy to obtain the $L_{x}^p(\mathbb{R}^2)$ boundedness of $\widetilde{T}_{j}$ uniformly in $u\in [1,2]$ and $j\in \mathbb{Z}_0^{-}$. It is the desired estimate on the closed line $[O, A]$ in Figure \ref{Figure:2}. By interpolation, it suffices to prove the uniform $L_x^{3}(\mathbb{R}^2)\rightarrow L_{x}^{3/2}(\mathbb{R}^2)$ estimate for $\widetilde{T}_{j}$, i.e., the estimate at point $D$ in Figure \ref{Figure:2}. We will use Lemma \ref{lemma 2.2} to complete the proof. Indeed, from \cite[Remark 1.4]{LSY} or \cite[Lemma 2.1]{LYu}, we have the following doubling condition,
\begin{align}\label{eq:02.2}
e^{C^{(1)}_{1}/2}\leq \frac{\gamma(2t)}{\gamma(t)}\leq e^{C^{(1)}_{2}}
\end{align}
for any $t\in (0,1]$, which will also be used in obtaining the necessity of $1+(1 +\omega)(\frac{1}{q}-\frac{1}{p})\geq0$. By a change of variables, $\widetilde{T}_{j}$ can be  pointwise bounded by $T^{\phi}_{e}$ with $I:=[\frac{1}{2}, 4]$ and $\phi(t):=u\Gamma_{j}\left(\frac{t}{u}\right)$, uniformly in $u\in [1,2]$ and $j\in \mathbb{Z}_0^{-}$. From \eqref{eq:02.2} and the following $(\textrm{iii})$ of Lemma \ref{lemma 3.1}, we have the uniform estimate that $|\phi''(t)|\gtrsim 1$ for any $t\in I$. In view of Lemma \ref{lemma 2.2}, we obtain
\begin{align}\label{eq:02.3}
\|\widetilde{T}_{j}f\|_{L^3(\mathbb{R}^2)} \lesssim \|f\|_{L^{3/2}(\mathbb{R}^{2})}
\end{align}
holds uniformly in $u\in [1,2]$ and $j\in \mathbb{Z}_0^{-}$. As a result, we finish the proof of the sufficiency parts in Theorem \ref{thm2}.

We now show the necessity. From Remark \ref{remark 3} and Theorem \ref{thm2}, since the estimate is uniform in $u\in [1,2]$, we just need to show\footnote{Note: The necessity of $\frac{1}{q}\leq\frac{1}{p}$ can be obtained by a theorem of $\textrm{H}\ddot{\textrm{o}}\textrm{rmander}$ \cite{Ho}. The necessity of $\frac{1}{q}\geq\frac{2}{p}-1$ can also be established by the adjoint operator $T_u^*$ of $T_uf(x):=Tf(x,u)$, which is essentially the operator $T_u$.} the necessity of $\frac{1}{2p}\leq\frac{1}{q}$. Indeed, let $f:=\chi_{A}$, where $A$ is the $\epsilon$ neighborhood of the curve $(-ut,-u\gamma(t))$ with $t\in (0,1]$, $u\in [1,2]$ and $\epsilon>0$ small enough. Then
\begin{align*}
      \int_{0}^{1}\chi_{A}(x_1-ut,x_2-u\gamma(t))\,\textrm{d}t= 1
      \end{align*}
for any $(x_1,x_2)\in B(0, \epsilon)$. It further follows that
      \begin{align*}
      \left\|\int_{0}^{1}\chi_{A}(x_1-ut,x_2-u\gamma(t))\,\textrm{d}t\right\|_{L_x^{q}(B(0, \epsilon))}\gtrsim \epsilon^{\frac{2}{q}} ~~\textrm{and}~~ \|f\|_{L_x^{p}(\mathbb{R}^{2})}\approx \epsilon^{\frac{1}{p}}.
      \end{align*}
Apply the estimate \eqref{eq:02.1} to deduce that $\epsilon^{\frac{2}{q}-\frac{1}{p}}\lesssim 1$. As a consequence, we have that $\frac{1}{2p}\leq\frac{1}{q}$ since $\epsilon>0$ may tend to $0$. This concludes the necessity parts in Theorem \ref{thm3}.

\section{Proof of Theorem \ref{thm1}}

In this section, we devote to the proof of Theorem \ref{thm1}. Let us begin by introducing the following lemma.

\begin{lemma}\label{lemma 3.1}
(\cite[Lemma 2.2]{LSY}) Let $t\in [\frac{1}{2},2]$ and $\Gamma_{j}(t)=\frac{\gamma(2^jt)}{\gamma(2^j)}$ with $j\in \mathbb{Z}_0^{-}$. The following inequalities hold uniformly in $j$,
\begin{enumerate}
  \item[\rm(i)] $  e^{-C^{(1)}_{2}}\leq\Gamma_{j}(t)\leq e^{C^{(1)}_{2}}$;
  \item[\rm(ii)] $ \frac{C^{(1)}_1}{2e^{C^{(1)}_{2}}}\leq|\Gamma_{j}'(t)|\leq 2e^{C^{(1)}_{2}}C^{(1)}_2$;
  \item[\rm(iii)] $\frac{C^{(2)}_1}{4e^{C^{(1)}_{2}}}\leq|\Gamma_{j}''(t)|\leq 4e^{C^{(1)}_{2}}C^{(2)}_{2}$;
  \item[\rm(iv)] $|\Gamma_{j}^{(k)}(t)|\leq  2^{k} e^{C^{(1)}_{2}}C^{(k)}_{2}$ \quad \textrm{for} \textrm{all} $2\leq k\leq N$ \textrm{and} $k\in \mathbb{N}$;
  \item[\rm(v)] $|((\Gamma_{j}')^{-1})^{(k)}(t)|\lesssim 1$ \quad \textrm{for} \textrm{all} $0\leq k< N$  \textrm{and} $k\in \mathbb{N}$, \textrm{where} $(\Gamma_{j}')^{-1}$ \textrm{is} \textrm{the} \textrm{inverse} \textrm{function} \textrm{of} $\Gamma_{j}'$.
\end{enumerate}
\end{lemma}

We now turn to the proof of Theorem \ref{thm1}. Similarly as in the proof of Theorem \ref{thm3}, using the fact that
 \begin{align*}
 \delta_j\left( \left\|T_{j}f\right\|_{L_{u}^q( [1,2])}\right)=2^j \left\|\widetilde{T}_{j}(\delta_jf)\right\|_{L_{u}^q( [1,2])}
  \end{align*}
holds for all $j\in \mathbb{Z}_0^{-}$, it is enough to show
\begin{align*}
\sum_{j\in \mathbb{Z}_0^{-}} 2^j|2^j\gamma(2^j)|^{\frac{1}{q}-\frac{1}{p}} \left\| \widetilde{T}_{j} \right\|_{L_x^{p}(\mathbb{R}^{2})\rightarrow L_{x,u}^q(\mathbb{R}^2\times [1,2])}\lesssim 1.
\end{align*}
From Lemma \ref{lemma 2.1}, we should only obtain $\widetilde{T}_{j} :\ L_x^{p}(\mathbb{R}^{2})\rightarrow L_{x,u}^q(\mathbb{R}^2\times [1,2])$ uniformly in $j\in \mathbb{Z}_0^{-}$ for all $(\frac{1}{p},\frac{1}{q}) \in \square\cup \{(0,0)\}\cup \{(\frac{2}{3},\frac{1}{3})\}$. By interpolation, it suffices to establish the corresponding uniform estimate with $(\frac{1}{p},\frac{1}{q})$ lies in Figure \ref{Figure:3}. For the closed line $[O, A]$, by Minkowski's inequality, we may obtain the $ L_x^{p}(\mathbb{R}^{2})\rightarrow L_{x,u}^q(\mathbb{R}^2\times [1,2])$ estimate for $\widetilde{T}_{j}$ uniformly in $j\in \mathbb{Z}_0^{-}$ for all $p=q\in [1,\infty]$. For the point $D(\frac{2}{3},\frac{1}{3})$, $L_u^q([1,2])$ integration in \eqref{eq:02.3} will leads to the uniform estimate $\widetilde{T}_{j} :\ L_x^{3/2}(\mathbb{R}^2)\rightarrow L_{x,u}^3(\mathbb{R}^2\times[1,2])$ as desired. Therefore, it suffices to consider the case of the half open line $(C, M]$ in Figure \ref{Figure:3}.
\begin{figure}[htbp]
  \centering
  \includegraphics[width=3.8in]{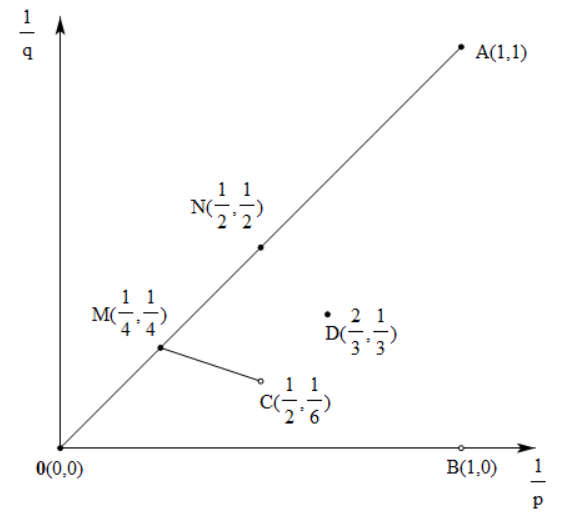}\\
  \caption{The framework to obtain the uniform estimate $\widetilde{T}_{j} :\ L_x^{p}(\mathbb{R}^{2})\rightarrow L_{x,u}^q(\mathbb{R}^2\times [1,2])$.}\label{Figure:3}
\end{figure}

After taking a Fourier transform in \eqref{eq:02.02}, let $\xi:=(\xi_1, \xi_2)$, we see that
$$\widetilde{T}_{j}f(x,u)=\int_{\mathbb{R}^{2}} \hat{f}(\xi) e^{ix\cdot\xi}H_j(u, \xi)\,\textrm{d}\xi ~~~~~ \textrm{with}~~~~~ H_j(u, \xi):=\int_{\mathbb{R}}  e^{-iu\xi_1 t-i u \xi_2 \Gamma_j(t) }\psi(t)\,\textrm{d}t.$$
Let $\Psi^0:= \sum_{k\leq 0}\psi_k$, we decompose
\begin{align}\label{eq:3.1}
\widetilde{T}_{j}f(x,u)= \widetilde{T}^0_{j}f(x,u)+ \sum_{k\geq 1} \widetilde{T}_{j,k}f(x,u)
\end{align}
and try to estimate the each decomposed pieces, where
\begin{eqnarray*}
\left\{\aligned
\widetilde{T}^0_{j}f(x,u):= \int_{\mathbb{R}^{2}} \hat{f}(\xi) e^{ix\cdot\xi}H_j(u, \xi)\Psi^0(u|\xi|)\,\textrm{d}\xi;\\
\widetilde{T}_{j,k}f(x,u):=\int_{\mathbb{R}^{2}} \hat{f}(\xi) e^{ix\cdot\xi}H_j(u, \xi)\psi_k(u|\xi|)\,\textrm{d}\xi.
\endaligned\right.
\end{eqnarray*}

Consider $\widetilde{T}^0_{j}$. Noting that $\Psi^0$ is a compactly supported smooth function, by Lemma \ref{lemma 3.1}, we have $|\partial_u^{\alpha}\partial^{\beta}_{\xi}(H(u, \xi)\Psi^0(u|\xi|))|\lesssim 1/(1+|\xi|)^3$
for all $\alpha\in  \mathbb{N}^2_0$ with $|\alpha|\leq 3$, where the implicit constant is independent of $u\in[1,2]$ and  $j\in \mathbb{Z}_0^{-}$, which further leads to $(H(u, \cdot)\Psi^0(u|\cdot|))^{\vee}\in L^{r}(\mathbb{R}^{2})$
for all $1\leq r \leq \infty$. Thus, by Young's inequality, we assert that\footnote{Here and hereafter, our estimates are all uniformly in $j\in \mathbb{Z}_0^{-}$, we will not repeat it again.}
\begin{align*}
\left\|\widetilde{T}^0_{j}f\right\|_{L_{x,u}^q(\mathbb{R}^2\times [1,2])} \lesssim\|f\|_{L_x^{p}(\mathbb{R}^{2})}
\end{align*}
for all $q\geq p \geq 1$.

Consider $\widetilde{T}_{j,k}$. The phase function in $H_j(u, \xi)$ is $\varphi_j(u,\xi,t):=-u\xi_1 t- u \xi_2 \Gamma_j(t)$. Differentiate in $t$ to obtain
\begin{align*}
\partial_t\varphi_j(u,\xi,t)=-u\xi_1 - u \xi_2 \Gamma'_j(t), ~~\partial^2_t\varphi_j(u,\xi,t)=- u \xi_2 \Gamma''_j(t), ~~\partial^2_t\varphi_j(u,\xi,t)=- u \xi_2 \Gamma'''_j(t).
\end{align*}
By $(\textrm{ii})$ of Lemma \ref{lemma 3.1}, if $|\xi_1|\geq 10 e^{C^{(1)}_{2}}C^{(1)}_2 |\xi_2|$ or $|\xi_2|\geq 10e^{C^{(1)}_{2}}/ C^{(1)}_{1}|\xi_1|$, we have $|\partial_t\varphi_j(u,\xi,t)|\gtrsim |\xi|$. Let $\rho \in C^{\infty}_{c}(\mathbb{R}^{+})$ be a function such that $\rho =1$ on $[C^{(1)}_{1}/ (10e^{C^{(1)}_{2}}), 10 e^{C^{(1)}_{2}}C^{(1)}_2]$. Observe that $|\xi|\approx 2^k$, integration by parts shows that $|\partial_u\partial^{\alpha}_{\xi} [(1-\rho(\frac{|\xi_1|}{|\xi_2|}))H(u, \xi)\psi_k(u|\xi|)]|\lesssim 1/ (1+|\xi|)^3$ for all $|\alpha|\leq 3$, where the implicit constant is independent of $u\in[1,2]$. We now define
\begin{eqnarray*}
\left\{\aligned
&\widetilde{T}^1_{j,k}f(x,u):=\int_{\mathbb{R}^{2}} \hat{f}(\xi) e^{ix\cdot\xi}\rho\left(\frac{|\xi_1|}{|\xi_2|}\right)H_j(u, \xi)\psi_k(u|\xi|)\,\textrm{d}\xi;\\
&\widetilde{T}^2_{j,k}f(x,u):=\int_{\mathbb{R}^{2}} \hat{f}(\xi) e^{ix\cdot\xi}\left(1-\rho\left(\frac{|\xi_1|}{|\xi_2|}\right)\right)H_j(u, \xi)\psi_k(u|\xi|)\,\textrm{d}\xi.
\endaligned\right.
\end{eqnarray*}
Then we conclude immediately the desired estimate that
\begin{align}\label{eq:3.2}
\left\|\widetilde{T}^2_{j,k}f\right\|_{L_{x,u}^q(\mathbb{R}^2\times [1,2])} \lesssim 2^{-k} \|f\|_{L_x^{p}(\mathbb{R}^{2})}
\end{align}
for all $q\geq p \geq 1$. This, combined with \eqref{eq:3.1}, implies that we should only to bound the operator $\widetilde{T}^1_{j,k}$ with a decay in $k$.

Consider $\widetilde{T}^1_{j,k}$. If $\xi_1\xi_2\geq0$, it is easy to see that $|\partial_t\varphi_j(u,\xi,t)|\gtrsim |\xi|$. By the argument as above, we also obtain the desired estimate \eqref{eq:3.2} for $\widetilde{T}^1_{j,k}$. Therefore, from now on, we restrict our attention to the most difficult case of $\xi_1\xi_2<0$. Let
$$\partial_t\varphi_j(u,\xi,t_0)=-u\xi_1 - u \xi_2 \Gamma'_j(t_0)=0.$$
Note that $t_0$ is the critical point and $\Gamma'_j(t_0)=-\frac{\xi_1}{\xi_2}$. Since our estimates about $t_0$ is uniform in $j$, we omit the parameter $j$ in the notation $t_0$. Since $\Gamma'_j$ is strictly monotonic, we can write $t_0=(\Gamma'_j)^{-1}(-\frac{\xi_1}{\xi_2})$. We can further assume that $t_0\in (1/2,2)$. Otherwise, it is easy to see that $|\varphi'_t(u,\xi,t)|\gtrsim |\xi|$ holds, which leads to the desired estimate as above.

We consider a one-dimensional oscillatory integral involving a phase function with a non-degenerate critical point. Based on the spirit of stationary phase estimates, we rewrite
\begin{align*}
H_j(u, \xi)=\int_{\mathbb{R}}  e^{i\varphi_j(u,\xi,t+t_0) }\psi(t+t_0)\,\textrm{d}t.
\end{align*}
By Taylor's theorem and the fact that $t_0$ is the critical point, we conclude that
\begin{align*}
\varphi_j(u,\xi,t+t_0)=\varphi_j(u,\xi,t_0)-u\xi_2t^2\eta_j(t,t_0),
\end{align*}
where $\eta_j(t,t_0):=\frac{\Gamma''_j(t_0)}{2}+\frac{t}{2}\int_0^1(1-\theta)^2\Gamma'''_j(\theta t+t_0)\,\textrm{d}\theta$, which can be viewed as a compactly supported smooth function. Therefore,
\begin{align}\label{eq:3.3}
\widetilde{T}^1_{j,k}f(x,u)=\int_{\mathbb{R}^{2}} e^{ix\cdot\xi}e^{i \varphi_j(u,\xi,t_0)}m_{j,k}(u,\xi) \hat{f}(\xi)\,\textrm{d}\xi,
\end{align}
where $$m_{j,k}(u,\xi):=\rho\left(\frac{|\xi_1|}{|\xi_2|}\right)\psi_k(u|\xi|)  \int_{\mathbb{R}}  e^{-iu\xi_2t^2\eta_j(t,t_0) }\psi(t+t_0)\,\textrm{d}t.$$

\textbf{Localization}.

$\widetilde{T}^1_{j,k}$ can be localized. Indeed, we write
\begin{align*}
\widetilde{T}^1_{j,k}f(x,u)=\int_{\mathbb{R}^{2}}K_{j,k}(x,u,y) f(y) \,\textrm{d}y
\end{align*}
with the kernel
\begin{align*}
K_{j,k}(x,u,y):=\int_{\mathbb{R}^{2}} e^{i(x-y)\cdot\xi}e^{i \varphi_j(u,\xi,t_0)}m_{j,k}(u,\xi) \,\textrm{d}\xi.
\end{align*}
Notice that $\Gamma'_j(t_0)=-\frac{\xi_1}{\xi_2}$, $t_0\approx 1$ and $\Gamma_j(t_0)\approx 1$, it follows that there exists a positive constant $\varpi$ large enough such that $|\nabla_{\xi}[(x-y)\cdot\xi+\varphi_j(u,\xi,t_0)]|\gtrsim |x-y|$
if $|x-y|\geq \varpi$. Therefore, via an integration by parts, we can bound the kernel $K_{j,k}$
by $2^{(-\frac{1}{2}-N)k}/ |x-y|^N$ uniformly in $u\in [1,2]$ for some $N\in \mathbb{N}$ large enough and $|x-y|\geq \varpi$. Furthermore, we may obtain the desired estimate that
\begin{align*}
\left\|\widetilde{T}^{1,a}_{j,k}f\right\|_{L_{x,u}^q(\mathbb{R}^2\times [1,2])}\lesssim 2^{(-\frac{1}{2}-N)k} \|f\|_{L_x^{p}(\mathbb{R}^{2})}
\end{align*}
for all $q\geq p \geq 1$, where
\begin{align*}
\widetilde{T}^{1,a}_{j,k}f(x,u):=\int_{\mathbb{R}^{2}} \chi_{B^{\complement}(x,\varpi)} (y) K_{j,k}(x,u,y)  f(y) \,\textrm{d}y.
\end{align*}
As a result, we should only to consider the following localized operator,
\begin{align}\label{eq:3.4}
\widetilde{T}^{1,b}_{j,k}f(x,u):=\int_{\mathbb{R}^{2}} \chi_{B(x,\varpi)} (y)K_{j,k}(x,u,y) f(y) \,\textrm{d}y.
\end{align}

\textbf{Reduce to a local smoothing estimate}.

We start with some definitions. Let $\Omega:\ \mathbb{R}^2\times\mathbb{R}\rightarrow\mathbb{R}$ be a nonnegative smooth function with the properties that $\Omega(x,u)=1$ on $\{(x,u)\in \mathbb{R}^2\times\mathbb{R}:\ |x|\leq \varpi, u\in [1,2]\}$ and vanishing outside $\{(x,u)\in \mathbb{R}^2\times\mathbb{R}:\  |x|\leq 2\varpi, u\in(1/2, 5/2)\}$, $\Psi_j(x,u,\xi):=x\cdot\xi+\varphi_j(u,\xi,t_0)$ and $a_{j,k}(x,u,\xi):=\Omega(x,u)m_{j,k}(u,\xi)$. Define
\begin{align}\label{eq:3.5}
F_{j,k}f(x,u):=\int_{\mathbb{R}^{2}} e^{i\Psi_j(x,u,\xi)}a_{j,k}(x,u,\xi) \hat{f}(\xi)\,\textrm{d}\xi.
\end{align}
In order to estimate $\widetilde{T}^{1,b}_{j,k}$ in \eqref{eq:3.4}, it suffices to obtain the following estimate: there exists a positive constant $\varepsilon$ such that
\begin{align}\label{eq:3.6}
\left\|F_{j,k}f\right\|_{L_{x,u}^{q}(\mathbb{R}^{3})} \lesssim 2^{- \varepsilon k} \|f\|_{L_x^{p}(\mathbb{R}^{2})}
\end{align}
for all $(\frac{1}{p},\frac{1}{q})$ in the half open line $(C, M]$ in Figure \ref{Figure:3}.

Indeed, we can decompose $\mathbb{R}^{2}=\bigcup_{i\in \mathbb{Z}^{2}}B(x_i,\varpi)$ such that $|x_i-x_{i'}|\approx |i-i'|\varpi$ for any $i\neq i'$. Next, we change variables to bound $\|\widetilde{T}^{1,b}_{j,k}f\|^q_{L_{x,u}^{q}(\mathbb{R}^{2}\times[1,2])}$ by
\begin{align*}
\sum_{i\in \mathbb{Z}^{2}}\left\|\int_{\mathbb{R}^{2}} \chi_{B(x,\varpi)} (y)K_{j,k}(x,u,y) f(y-x_i) \,\textrm{d}y\right\|^q_{L_{x,u}^{q}(B(0,\varpi)\times [1,2] )},
\end{align*}
which is further bounded by
\begin{align}\label{eq:3.7}
\sum_{i\in \mathbb{Z}^{2}}\left\|(E_{j,k}-F_{j,k})(f(\cdot-x_i))\right\|^q_{L_{x,u}^{q}(\mathbb{R}^{3} )}+\sum_{i\in \mathbb{Z}^{2}}\left\|F_{j,k}(f(\cdot-x_i))\right\|^q_{L_{x,u}^{q}(\mathbb{R}^{3} )},
\end{align}
where
\begin{align*}
E_{j,k}f(x,u):=\Omega(x,u)\int_{\mathbb{R}^{2}} \chi_{B(x,\varpi)} (y)K_{j,k}(x,u,y) f(y) \,\textrm{d}y.
\end{align*}
The operator $E_{j,k}-F_{j,k}$ can be handled by a similar argument to $\widetilde{T}^{1,a}_{j,k}$ since the kernel of $E_{j,k}-F_{j,k}$ is supported on $\{(x,y)\in\mathbb{R}^{2}:\ |x-y|\geq \varpi\}$. Therefore, the first part of \eqref{eq:3.7} can be controlled by
\begin{align*}
 \sum_{i\in \mathbb{Z}^{2}}2^{(-\frac{1}{2}-N)k q} \|f(\cdot-x_i)  \chi_{B(0,2\varpi)} (\cdot)\|^q_{L_{x}^{p}(\mathbb{R}^{2})}
\lesssim 2^{(-\frac{1}{2}-N)k q} \|f\|^q_{L_x^{p}(\mathbb{R}^{2})}
\end{align*}
for some $N\in \mathbb{N}$ large enough, as desired. For the second part of \eqref{eq:3.7}, by \eqref{eq:3.6}, we can also obtain the following desired estimate,
\begin{align*}
\sum_{i\in \mathbb{Z}^{2}}2^{-\varepsilon k q}\|f(\cdot-x_i)  \chi_{B(0,2\varpi)} (\cdot)\|^q_{L_{x}^{p}(\mathbb{R}^{2})}\lesssim 2^{-\varepsilon k q} \|f\|^q_{L_x^{p}(\mathbb{R}^{2})}.
\end{align*}
Putting things together, by \eqref{eq:3.7}, we can establish $\widetilde{T}^{1,b}_{j,k} :\ L_x^{p}(\mathbb{R}^{2})\rightarrow L_{x,u}^q(\mathbb{R}^2\times [1,2])$ with bound $2^{(-\frac{1}{2}-N)k}+2^{-\varepsilon k }$, this is the desired estimate. Thus, we can reduce our considerations to proving a local smoothing estimate in \eqref{eq:3.6}.

\textbf{A key local smoothing estimate}.

Consider the \emph{Fourier integral operators}
\begin{align}\label{eq:3.a}
\mathcal{F}f(z):=\int_{\mathbb{R}^{n}} e^{i\phi(z,\xi)}a(z,\xi) \hat{f}(\xi)\,\textrm{d}\xi,
\end{align}
where $a(z,\xi)$ is a symbol of order $\sigma$. Here and hereafter, $z$ is used to denote a vector in $\mathbb{R}^{n}\times \mathbb{R}$ comprised of the space-time variables $(x,u)$. Suppose that $\textrm{supp}~a(\cdot,\xi)$ is contained in a fixed compact set and $\phi(z,\cdot)$ is a homogeneous function of degree $1$. We say that $\mathcal{F}$ satisfies the \emph{cinematic curvature condition} if the following conditions are satisfied:
\begin{enumerate}
 \item[$\bullet$] \textbf{Non-degeneracy condition.} ~~For all $(z,\xi)\in \textrm{supp}~a$
 \begin{align*}
 \textrm{rank}~\partial^2_{z\xi}\phi(z,\xi)=n.
\end{align*}
 \item[$\bullet$] \textbf{Cone condition.} ~~Consider the Gauss map $G :\ \textrm{supp}~a \rightarrow \mathbb{S}^n$ by $G(z,\xi):=\frac{G_0(z,\xi)}{|G_0(z,\xi)|}$, where
  \begin{align*}
     G_0(z,\xi):=\bigwedge_{j=1}^{n} \partial_{\xi_j}\partial_z\phi(z,\xi).
      \end{align*}
The curvature condition
\begin{align*}
\textrm{rank}~\partial^2_{\xi\xi}\langle\partial_z\phi(z,\xi),G(z,\xi_0)\rangle |_{\xi=\xi_0}=n-1
\end{align*}
holds for all $(z,\xi_0)\in \textrm{supp}~a$.
\end{enumerate}
Geometrically, the curvature condition of course means that, for fixed $z_0$, the cone $$\Gamma_{z_0}:=\left\{\partial_z\phi(z_0,\eta) :\ \eta\in \mathbb{R}^{n}\backslash \{0\} ~\textrm{in}~ \textrm{a} ~\textrm{conic} ~\textrm{neighbourhood} ~\textrm{of}~ \eta_0 \right\}$$ is a smooth conic manifold of dimension $n$ with $n-1$ non-vanishing principal curvatures at every point. We can see that the phase featured in the half-wave propagator, i.e., $\phi(z,\xi)=x\cdot\xi-u|\xi|$, satisfies the cinematic curvature condition.

Now, we show some local smoothing estimates for $\mathcal{F}$ in \eqref{eq:3.a}, which will be used to prove \eqref{eq:3.6}. To state the following propositions we introduce some definitions. Let $\bar{\psi}:\ \mathbb{R}^+\rightarrow\mathbb{R}$ be a smooth function supported on $\{t\in \mathbb{R}:\ \frac{1}{8}\leq t\leq 8\}$ with $0\leq \psi(t)\leq 1$ and $\bar{\psi}(t)=1$ on $\{t\in \mathbb{R}:\ \frac{1}{4}\leq t\leq 4\}$. We then use $P_k$ to denote the operator on $\mathbb{R}^{n}$ with multiplier $\bar{\psi}_k(|\xi|)$, where $\bar{\psi}_k(|\xi|):=\bar{\psi} (2^{-k}|\xi|)$ and $k\in \mathbb{N}$.

\begin{proposition}\label{prop 3.1}
Let $1\leq p \leq q \leq \infty$ and $s\in \mathbb{R}$. Let $\mathcal{F}$ be the Fourier integral operators in \eqref{eq:3.a} with symbol of order $\sigma$ and satisfy the cinematic curvature condition, together with the special assumption that $\textrm{rank}~\partial^2_{\xi\xi}\phi(z,\xi)=n-1$ when $\xi\neq 0$. Then, if $q=p\in [1,\infty)$, we require $s\geq (n-1)|\frac{1}{p}-\frac{1}{2}|+\sigma$; if $q=\infty$, we require $s\geq \frac{n-1}{2}+\frac{1}{p}+\sigma$, there exists a positive constant $C_{(u)}$ such that
\begin{align}\label{eq:3.b}
\left\|\mathcal{F}P_kf\right\|_{L_{x}^q(\mathbb{R}^n)}\leq C_{(u)} \|f\|_{W_x^{s,p}(\mathbb{R}^{n})}.
\end{align}
Note that the case $1\leq p <q <\infty$ can be proved by interpolation, we omit the details.
\end{proposition}

\begin{proof}[Proof of Proposition \ref{prop 3.1}.] For the case $q=p\in [1,\infty)$, which follows from the fixed time estimate in Seeger, Sogge and Stein \cite{SSS}. Here, for the case $p=q=1$, we used the estimate $\|P_kf\|_{H_{x}^1(\mathbb{R}^n)}\lesssim \|f\|_{L_x^{1}(\mathbb{R}^{n})}$. For the case $q=\infty$, by interpolation, it suffices to show
\begin{align}\label{eq:3.d}
\left\|\mathcal{F}P_kf\right\|_{L_{x}^{\infty}(\mathbb{R}^n)}\leq C_{(u)} \left\|f\right\|_{W_x^{\frac{n+1}{2}+\sigma,1}(\mathbb{R}^{n})}
\end{align}
and
\begin{align}\label{eq:3.c}
\left\|\mathcal{F}P_kf\right\|_{L_{x}^{\infty}(\mathbb{R}^n)}\leq C_{(u)} \left\|f\right\|_{W_x^{\frac{n-1}{2}+\sigma,\infty}(\mathbb{R}^{n})}.
\end{align}
\eqref{eq:3.d} follows from Stein \cite[Chapter IX, 6.14]{Stein} and the estimate $\|P_kf\|_{H_{x}^1(\mathbb{R}^n)}\lesssim \|f\|_{L_x^{1}(\mathbb{R}^{n})}$. We now prove \eqref{eq:3.c}. First of all, Seeger, Sogge and Stein \cite[Theorem 2.2]{SSS} obtained the $H_x^1(\mathbb{R}^n)\rightarrow L_{x}^1(\mathbb{R}^n)$ estimate for the Fourier integral operators with symbol of order $-\frac{n-1}{2}$. From Stein \cite[Chapter IX, Section 4]{Stein}, this result also holds for its adjoint operator, which further leads to the corresponding $L_x^{\infty}(\mathbb{R}^n)\rightarrow BMO_{x}(\mathbb{R}^n)$ estimate for these Fourier integral operators with symbol of order $-\frac{n-1}{2}$. Consequently
\begin{align}\label{eq:3.z}
\left\|\mathcal{F}P_kf\right\|_{BMO_{x}(\mathbb{R}^n)}\leq C_{(u)} \left\|f\right\|_{W_x^{\frac{n-1}{2}+\sigma,\infty}(\mathbb{R}^{n})}.
\end{align}
Notice that $\|\mathcal{F}P_kf\|_{L_{x}^{\infty}(\mathbb{R}^n)}$ can be bounded by
\begin{align*}
 \left\|\tilde{P}_k\mathcal{F}P_kf\right\|_{L_{x}^{\infty}(\mathbb{R}^n)}+ \left\|(I-\tilde{P}_k)\mathcal{F}P_kf\right\|_{L_{x}^{\infty}(\mathbb{R}^n)}\leq C_{(u)}\left( \left\|\mathcal{F}P_kf\right\|_{BMO_{x}(\mathbb{R}^n)}+\left\|f\right\|_{W_x^{\frac{n-1}{2}+\sigma,\infty}(\mathbb{R}^{n})}\right),
\end{align*}
where $\tilde{P}_kf:=(\tilde{\psi}_k(|\cdot|) \hat{f}(\cdot))^{\vee}$, $\tilde{\psi}$ is a compactly supported $C^{\infty}$ function on $\mathbb{R}\backslash \{0\}$ and $\tilde{\psi}=1$ on $[C^{-1}, C]$ for a positive constant $C$ sufficiently large. This, combined with \eqref{eq:3.z}, leads to \eqref{eq:3.c}. Therefore, the proof is complete.
\end{proof}

We give the local smoothing estimates for $\mathcal{F}$ about $P_kf$ in the following Proposition \ref{prop 3.2}. Since Theorem \ref{thm1} is almost sharp and relies heavily on this proposition, we may see that the result in  Proposition \ref{prop 3.2} is also almost sharp.

\begin{proposition}\label{prop 3.2}
Let $n=2$, $1\leq p \leq q \leq \infty$ and $s\in \mathbb{R}$. Let $\mathcal{F}$ be the Fourier integral operators in \eqref{eq:3.a} with symbol of order $\sigma$ and satisfy the cinematic curvature condition, together with the special assumption that $\textrm{rank}~\partial^2_{\xi\xi}\phi(z,\xi)=1$ when $\xi\neq 0$. If $q=p\in [1,\infty)$, we require: $s\geq\frac{1}{p}-\frac{1}{2}+\sigma$ for $p\in [1,2]$; $s> \sigma$ for $p\in (2,4)$; $s> \frac{1}{2}-\frac{2}{p}+\sigma$ for $p\in [4,\infty)$. If $q=\infty$, we require $s\geq\frac{1}{2}+\frac{1}{p}+\sigma$. Then, there exists a positive constant $C$ such that
\begin{align}\label{eq:3.e}
\left\|\mathcal{F}P_kf\right\|_{L_{x,u}^q(\mathbb{R}^2\times [1,2])}\leq C \|f\|_{W_x^{s,p}(\mathbb{R}^{2})}.
\end{align}
Note that the case $1\leq p <q <\infty$ can be obtained by interpolation, we omit the details.
\end{proposition}

\begin{proof}[Proof of Proposition \ref{prop 3.2}.] For the case $q=p\in [1,2]$, since $C_{(u)}\lesssim 1$ for any $u\in[1,2]$ from Seeger, Sogge and Stein \cite{SSS}, it follows that \eqref{eq:3.e} can be obtained by Proposition \ref{prop 3.1}. For the case $q=p\in [4,\infty)$, \eqref{eq:3.e} follows from Gao, Liu, Miao and Xi \cite[Theorem 1.4]{GLMX}, which is a generalization of the local smoothing conjecture in $2+1$ dimensions obtained in Guth, Wang and Zhang \cite[Theorem 1.2]{GWZ}. For the case $q=p\in (2,4)$, by interpolation, we may establish \eqref{eq:3.e}. It is easy to see that \eqref{eq:3.e} also follows from Proposition \ref{prop 3.1} for the case $q=\infty$, since $C_{(u)}\lesssim 1$ for any $u\in[1,2]$. This finishes the proof.
\end{proof}

In order to obtain \eqref{eq:3.6} by the key Proposition \ref{prop 3.2}. We need to verify that the Fourier integral operators $F_{j,k}$ in \eqref{eq:3.5} satisfy all of the conditions stated about $\mathcal{F}$ and the special assumption in Proposition \ref{prop 3.2}:
\begin{enumerate}
 \item[$\circ$] We can directly check that $\textrm{supp}~a_{j,k} \subset \{x\in \mathbb{R}^2:\ |x|\leq 2\varpi\}\times[1/2, 5/2]\times\{\xi\in \mathbb{R}^2:\ |\xi_1|\approx|\xi_2|\approx2^{k}\}$ and the phase function $\Psi_j(x,u,\cdot)$ is a homogeneous function of degree $1$.
 \item[$\circ$] $a_{j,k}(x,u,\xi)$ is a symbol of order $-\frac{1}{2}$. Indeed, by a similar argument in Stein \cite[P. 335, Proposition 3]{Stein} or Sogge \cite[P. 44, Lemma 1.1.2]{So}, one may verify that
\begin{align}\label{eq:3.8}
\left| \int_{\mathbb{R}}  e^{-i\lambda t^2\eta_j(t,t_0) }  t^{2l} \varrho(t,t_0)\,\textrm{d}t\right|\lesssim |\lambda|^{-\frac{1}{2}-l},
\end{align}
where $|\lambda|\geq 1$, $l\in \mathbb{N}_0$ and $\varrho$ is a compactly supported smooth function. Here, we used the uniform properties in Lemma \ref{lemma 3.1}, and leave the details to the reader. Let $I_{j,l}$ be the integral in \eqref{eq:3.8} and
\begin{align*}
\zeta_j(\lambda,t_0):=\int_{\mathbb{R}}  e^{-i\lambda t^2\eta_j(t,t_0) }\psi(t+t_0)\,\textrm{d}t.
\end{align*}
A calculation gives that $\partial_{t_{0}}^{b}\zeta_j$ can be written as a finite sum of $\lambda^lI_{j,l}$ for all $b<N-2$. Then, \eqref{eq:3.8} leads to $|\partial_{t_{0}}^{b}\zeta_j|\lesssim |\lambda|^{-\frac{1}{2}}$ for all $b<N-2$. Furthermore, by the Leibniz rule, we can write $\partial_{\lambda}^{a}\partial_{t_{0}}^{b}\zeta_j$ as a finite sum of $\lambda^{-a}\lambda^lI_{j,l}$ for all $a,b\in \mathbb{N}$ and $b<N-2$. Therefore, for $|\lambda|\geq 1$, by \eqref{eq:3.8}, we deduce that
\begin{align}\label{eq:3.9}
\left| \partial_{\lambda}^{a}\partial_{t_{0}}^{b}\zeta_j(\lambda,t_0)\right|\lesssim |\lambda|^{-\frac{1}{2}-a}
\end{align}
for all $a,b\in \mathbb{N}$ and $b<N-2$. On the other hand, $t_0=(\Gamma'_j)^{-1}(-\frac{\xi_1}{\xi_2})$, $|\xi|\approx 2^{k}\geq 1$ and $\frac{|\xi_1|}{|\xi_2|}\approx 1$, together with $(\textrm{v})$ of Lemma \ref{lemma 3.1} yield $|\partial_{\xi}^{\alpha}t_{0}|\lesssim (1+|\xi|)^{-|\alpha|}$
for all $\alpha\in \mathbb{N}_0^2$ with $|\alpha|<N$. This, combined with \eqref{eq:3.9}, shows that
\begin{align*}
\left|\partial_{u}^{\beta}\partial_{\xi}^{\alpha}m_{j,k}(u,\xi)\right|\lesssim (1+|\xi|)^{-\frac{1}{2}-|\alpha|}
\end{align*}
for all $(\alpha,\beta)\in \mathbb{N}^2_0\times \mathbb{N}_0$ with $|\alpha|+|\beta|<N$. Consequently,
\begin{align*}
\left|\partial_{x}^{\gamma}\partial_{u}^{\beta}\partial_{\xi}^{\alpha}a_{j,k}(x,u,\xi)\right|
\lesssim(1+|\xi|)^{-\frac{1}{2}-|\alpha|}
\end{align*}
for all $(\alpha,\beta,\gamma)\in \mathbb{N}^2_0\times \mathbb{N}_0\times \mathbb{N}^2_0$ with $|\alpha|+|\beta|+|\gamma|<N$, where the implicit constant is independent of $j$ and $k$. Therefore, $a_{j,k}(x,u,\xi)$ is a symbol of order $-\frac{1}{2}$.
 \item[$\circ$] $F_{j,k}$ in \eqref{eq:3.5} satisfies the non-degeneracy condition with $n=2$. Recall that $\Psi_j(z,\xi)=x_1\xi_1+x_2\xi_2-u\xi_1 t_0- u \xi_2 \Gamma_j(t_0)$ and $t_0=(\Gamma'_j)^{-1}(-\frac{\xi_1}{\xi_2})$. Then
\begin{align*}
\partial^2_{z\xi}\Psi_j(z,\xi)=\left(
\begin{array}{ccc}
1~& ~0 ~& ~-t_0 \\
0~& ~1 ~& ~-  \Gamma_j(t_0)
\end{array}
\right).
\end{align*}
It is clear that the rank of the mixed Hessian of $\Psi_j$ is $2$.
  \item[$\circ$] For the cone condition with $n=2$, one observes that $\partial_z \Psi_j(z,\xi)=(\xi_1, \xi_2, -\xi_1 t_0-  \xi_2 \Gamma_j(t_0))$.
Let $G:=(G_1,G_2,G_3)\in \mathbb{S}^2$ be the direction (unique up to sign) satisfying $\nabla_\xi\langle\partial_z\Psi_j(z,\xi),G\rangle=0$, combining $\Gamma'_j(t_0)=-\frac{\xi_1}{\xi_2}$, we then have $G_1-G_3t_0=0$ and $G_2 -G_3\Gamma_j(t_0)=0$, which further leads to $G=G(z,\xi)=(t_0, \Gamma_j(t_0), 1)/[t_0^2+\Gamma_j(t_0)^2+1]^{1/2}$. On the other hand, by calculation, we obtain
\begin{eqnarray*}
\left\{\aligned
&\partial^2_{\xi_1 \xi_1} \langle\partial_z\Psi_j(z,\xi),G(z,\xi_0)\rangle~|_{~\xi=\xi_0}=-G_3(z,\xi_0)\partial_{\xi_1}t_0,\\
&\partial^2_{ \xi_1\xi_2}\langle\partial_z\Psi_j(z,\xi),G(z,\xi_0)\rangle~|_{~\xi=\xi_0}=-G_3(z,\xi_0)\partial_{\xi_2}t_0,\\
&\partial^2_{\xi_2\xi_1}\langle\partial_z\Psi_j(z,\xi),G(z,\xi_0)\rangle~|_{~\xi=\xi_0}= -G_3(z,\xi_0)\Gamma'_j(t_0) \partial_{\xi_1}t_0,\\
&\partial^2_{\xi_2 \xi_2} \langle\partial_z\Psi_j(z,\xi),G(z,\xi_0)\rangle~|_{~\xi=\xi_0}=-G_3(z,\xi_0)\Gamma'_j(t_0) \partial_{\xi_2}t_0.
\endaligned\right.
\end{eqnarray*}
As a consequence, it follows that
\begin{align*}
\partial^2_{\xi\xi}\langle\partial_z\Psi_j(z,\xi),G(z,\xi_0)\rangle~|_{~\xi=\xi_0}=\left(
\begin{array}{ccc}
-G_3(z,\xi_0)\partial_{\xi_1}t_0~& ~-G_3(z,\xi_0)\partial_{\xi_2}t_0 \\
-G_3(z,\xi_0)\Gamma'_j(t_0) \partial_{\xi_1}t_0~& ~-G_3(z,\xi_0)\Gamma'_j(t_0) \partial_{\xi_2}t_0
\end{array}
\right).
\end{align*}
Observe that $\partial_{\xi_1}t_0=-\frac{1}{\xi_2}\frac{1}{\Gamma''_j(t_0)}$ and $\partial_{\xi_2}t_0=\frac{\xi_1}{\xi_2^2}\frac{1}{\Gamma''_j(t_0)}$, one obtains
\begin{align*}
\partial^2_{ \xi_1\xi_2}\langle\partial_z\Psi_j(z,\xi),G(z,\xi_0)\rangle~|_{~\xi=\xi_0}=\partial^2_{\xi_2\xi_1}\langle\partial_z\Psi_j(z,\xi),G(z,\xi_0)\rangle~|_{~\xi=\xi_0},
\end{align*}
which leads to $\textrm{rank}~\partial^2_{\xi\xi}\langle\partial_z\Psi_j(z,\xi),G(z,\xi_0)\rangle~|_{~\xi=\xi_0}=1$. Thus the cone condition is satisfied.
 \item[$\circ$] The special assumption that $\textrm{rank}~\partial^2_{\xi\xi}\Psi_j(z,\xi)=1$ when $\xi\neq 0$ is also satisfied. This part of the proof is similar to the proof of the cone condition with $n=2$. Notice that $\Gamma'_j(t_0)=-\frac{\xi_1}{\xi_2}$, a similar calculation gives
\begin{align*}
\partial^2_{\xi\xi}\Psi_j(z,\xi)=\left(
\begin{array}{ccc}
-u\partial_{\xi_1}t_0~& ~-u \partial_{\xi_2}t_0 \\
-u \Gamma'_j(t_0)\partial_{\xi_1}t_0~& ~-u\Gamma'_j(t_0) \partial_{\xi_2}t_0
\end{array}
\right).
\end{align*}
Then, by the facts that $\partial_{\xi_1}t_0=-\frac{1}{\xi_2}\frac{1}{\Gamma''_j(t_0)}$ and $\partial_{\xi_2}t_0=\frac{\xi_1}{\xi_2^2}\frac{1}{\Gamma''_j(t_0)}$, we obtain that $\textrm{rank}~\partial^2_{\xi\xi}\Psi_j(z,\xi)=1$ when $\xi\neq 0$.
\end{enumerate}

Therefore, $F_{j,k}$ satisfies the conditions in Proposition \ref{prop 3.2} with $\sigma=-\frac{1}{2}$. By $F_{j,k}f=F_{j,k}P_kf$ for any $k\in \mathbb{N}$ and $j\in \mathbb{Z}_0^{-}$, which follows from the definition in \eqref{eq:3.5}, then Proposition \ref{prop 3.2} implies
\begin{align}\label{eq:3.001}
\left\|F_{j,k}f\right\|_{L_{x,u}^4(\mathbb{R}^3)}\lesssim 2^{\left(-\frac{1}{2}+\vartheta_1\right)k}\|f\|_{L_x^{4}(\mathbb{R}^{2})}
\end{align}
for any $\vartheta_1\in(0,\infty)$, and
\begin{align}\label{eq:3.002}
\left\|F_{j,k}f\right\|_{L_{x,u}^{\infty}(\mathbb{R}^3)}\lesssim 2^{k}\|f\|_{L_x^{1}(\mathbb{R}^{2})}.
\end{align}
Using interpolation between \eqref{eq:3.001} and \eqref{eq:3.002}, we obtain that
\begin{align}\label{eq:3.003}
\left\|F_{j,k}f\right\|_{L_{x,u}^6(\mathbb{R}^3)}\lesssim 2^{\vartheta_2 k}\|f\|_{L_x^{2}(\mathbb{R}^{2})}
\end{align}
for any $\vartheta_2\in(0,\infty)$. Furthermore, by interpolation between \eqref{eq:3.001} and \eqref{eq:3.003} with $\vartheta_1:=\vartheta_2:=\frac{1}{2}-\frac{1}{p}>0$, we obtain
\begin{align}\label{eq:3.004}
\left\|F_{j,k}f\right\|_{L_{x,u}^q(\mathbb{R}^3)}\lesssim 2^{-\left(\frac{1}{2}-\frac{1}{p}\right) k}\|f\|_{L_x^{p}(\mathbb{R}^{2})}
\end{align}
for all $(\frac{1}{p}, \frac{1}{q})$ lies in the line segment $(C(\frac{1}{2}, \frac{1}{6}), M(\frac{1}{4}, \frac{1}{4})]$ in Figure \ref{Figure:3}. This is the desired estimate \eqref{eq:3.6} with $\varepsilon:=\frac{1}{2}-\frac{1}{p} >0$. Therefore, we finish the proof of Theorem \ref{thm1}.

\begin{remark} We can also obtain some other endpoint estimates for $F_{j,k}$ on the trapezium $\square$. Indeed, by Proposition \ref{prop 3.2}, we have
\begin{align}\label{eq:3.005}
\left\|F_{j,k}f\right\|_{L_{x,u}^1(\mathbb{R}^3)}\lesssim \|f\|_{L_x^{1}(\mathbb{R}^{2})},
\end{align}
\begin{align}\label{eq:3.006}
\left\|F_{j,k}f\right\|_{L_{x,u}^2(\mathbb{R}^3)}\lesssim 2^{-\frac{1}{2}k}\|f\|_{L_x^{2}(\mathbb{R}^{2})},
\end{align}
and
\begin{align}\label{eq:3.007}
\left\|F_{j,k}f\right\|_{L_{x,u}^{\infty}(\mathbb{R}^3)}\lesssim \|f\|_{L_x^{\infty}(\mathbb{R}^{2})}.
\end{align}
Using interpolation between \eqref{eq:3.002} and \eqref{eq:3.006}, we obtain that
\begin{align}\label{eq:3.008}
\left\|F_{j,k}f\right\|_{L_{x,u}^3(\mathbb{R}^3)}\lesssim \|f\|_{L_x^{3/2}(\mathbb{R}^{2})}.
\end{align}
Therefore, the point $D(\frac{2}{3},\frac{1}{3})$ is also an endpoint of $\widetilde{T}_{j}$, but not $L_x^{3/2}(\mathbb{R}^{2})$ to $L_{x,u}^{3}(\mathbb{R}^2\times [1,2])$ estimate. As in the point $C(\frac{1}{2},\frac{1}{6})$, by interpolation between $D$ with $M$, we can obtain the corresponding estimates on $(D, M]$. This, combined with the estimates on $(C, M]$ and $[O, A]$, leads to the desired estimates \eqref{eq:3.6} except for the case of $[D, A)$. In this paper, we use other methods and obtain the estimate $\widetilde{T}_{j} :\ L_x^{3/2}(\mathbb{R}^{2})\rightarrow L_{x,u}^3(\mathbb{R}^2\times [1,2])$ by Lemma \ref{lemma 2.2}, which implies that our desired estimates can be established on $[D, A)$.
\end{remark}

\section{Proof of Theorem \ref{thm2}}

In this section, we show the necessity of the region of $(\frac{1}{p}, \frac{1}{q})$ in Theorem \ref{thm1}, i.e., Theorem \ref{thm2}. We split the proof into four parts.

\textbf{Proof of (i) of Theorem \ref{thm2}}.

Since $\omega=\inf_{\tau\in (0,1)}\sup_{t\in (0,\tau)}\frac{\ln|\gamma(t)|}{\ln t}$, then, for any $\epsilon>0$, there exists $\tau\in (0,1)$ such that $\sup_{t\in (0,\tau)}\frac{\ln|\gamma(t)|}{\ln t}\in [\omega,\omega+\epsilon)$. Moreover, there exists $t\in (0,\tau)$ such that $\frac{\ln|\gamma(t)|}{\ln t}\in (\omega-\epsilon,\omega+\epsilon)$, which further implies that
      $$\sup_{s\in (0,t)}\frac{\ln|\gamma(s)|}{\ln s}\leq\sup_{s\in (0,\tau)}\frac{\ln|\gamma(s)|}{\ln s}< \omega+\epsilon ~~\textrm{and}~~ \sup_{s\in (0,t)}\frac{\ln|\gamma(s)|}{\ln s}\geq \omega. $$
      Hence, there exists $t'\in (0,t)$ such that $\frac{\ln|\gamma(t')|}{\ln t'}\in (\omega-\epsilon,\omega+\epsilon)$. Repeating the same procedure, there exists a sequence $\{t_i\}_{i=1}^{\infty}\subset (0,1)$ such that $\frac{\ln|\gamma(t_i)|}{\ln t_i}\in (\omega-\epsilon,\omega+\epsilon)$, where $\{t_i\}_{i=1}^{\infty}$ is strictly decreasing and $\lim_{i\rightarrow\infty}t_i=0$ . Then we clearly have
      \begin{align}\label{eq:3.1a}
      \left|\gamma(t_i)\right|={t_i}^{\log_{t_i} {|\gamma(t_i)|}}={t_i}^{\frac{\ln |\gamma(t_i)|}{\ln t_i}}\in \left({t_i}^{\omega+\epsilon}, {t_i}^{\omega-\epsilon}\right).
      \end{align}

      On the other hand, we claim that
      \begin{align}\label{eq:3.1b}
      t_i\left(t_i\left|\gamma(t_i)\right|\right)^{\frac{1}{q}-\frac{1}{p}}\lesssim 1,
      \end{align}
      where the implicit constant is independent of $i$. Indeed, let $S_{t_i}:=[-t_i, t_i]\times[-|\gamma(t_i)|, |\gamma(t_i)|]$ and $$\Lambda_{\chi_{S_{t_i}}}(x,u):=\left|\left\{t\in [0,1]:\ (x_1-ut,x_2-u\gamma(t))\in S_{t_i}\right\}\right|.$$
       Then, for any $x\in S_{t_i}/2$, by the doubling condition \eqref{eq:02.2}, there exists $\delta\in (0,1)$ small enough such that, for any $t\in [0,\delta t_i]$, we have $|x_1-ut|\leq t_i/2+2\delta t_i\leq t_i$ and $|x_2-u\gamma(t)|\leq |\gamma(t_i)|/2+2|\gamma(\delta t_i)|\leq |\gamma(t_i)|$, which further leads to the pointwise estimate $\Lambda_{\chi_{S_{t_i}}}(x,u)\geq \delta t_i \chi_{S_{t_i}/2}(x)$. Noticing that $T(\chi_{S_{t_i}})(x,u)\geq \Lambda_{\chi_{S_{t_i}}}(x,u) $ and the assumption that $\|Tf\|_{L_{x,u}^q(\mathbb{R}^2\times [1,2])} \lesssim \|f\|_{L_x^{p}(\mathbb{R}^{2})}$, we conclude that
       \begin{align*}
      |S_{t_i}|\lesssim \left|\left\{(x,u)\in\mathbb{R}^2\times[1,2]:\ T(\chi_{S_{t_i}})(x,u)\geq \delta t_i \right\}\right|\lesssim \left(t_i^{-1} \left\|\chi_{S_{t_i}}\right\|_{L_x^{p}(\mathbb{R}^{2})} \right)^{q}.
      \end{align*}
Combining this with the fact that $|S_{t_i}|\approx t_i |\gamma(t_i)|$ and $\|\chi_{S_{t_i}}\|_{L_x^{p}(\mathbb{R}^{2})}\approx (t_i |\gamma(t_i)|)^{\frac{1}{p}}$, we obtain \eqref{eq:3.1b}.

A theorem of $\textrm{H}\ddot{\textrm{o}}\textrm{rmander}$ \cite{Ho} implies that $\frac{1}{q}\leq \frac{1}{p}$. Combining \eqref{eq:3.1a} and \eqref{eq:3.1b}, we obtain
      \begin{align*}
      1\gtrsim t_i\left(t_i\left|\gamma(t_i)\right|\right)^{\frac{1}{q}-\frac{1}{p}} \gtrsim t_i\left(t_i\left({t_i}^{\omega-\epsilon}\right)\right)^{\frac{1}{q}-\frac{1}{p}}={t_i}^{1+(1 +\omega-\epsilon)\left(\frac{1}{q}-\frac{1}{p}\right)}.
      \end{align*}
      Notice that the sequence $\{t_i\}_{i=1}^{\infty}\subset (0,1)$ is strictly decreasing and $\lim_{i\rightarrow\infty}t_i=0$, so that
      \begin{align}\label{eq:3.1c}
      1+(1 +\omega-\epsilon)\left(\frac{1}{q}-\frac{1}{p}\right)\geq 0
      \end{align}
      for all $\epsilon>0$. Let $\epsilon\rightarrow0$ in \eqref{eq:3.1c}, it follows that $1+(1 +\omega)(\frac{1}{q}-\frac{1}{p})\geq 0$, and so we have (i) of Theorem \ref{thm2}.

\textbf{Proof of (ii) of Theorem \ref{thm2}}.

As above, the condition $\frac{1}{q}\leq\frac{1}{p}$ follows from a theorem of $\textrm{H}\ddot{\textrm{o}}\textrm{rmander}$ \cite{Ho}. Hence, we should only to show that $\frac{2}{p}-1\leq\frac{1}{q}$. For any $\epsilon>0$, let $D$ be the $\epsilon$ neighborhood of the curve $(ut,u\gamma(t))$ with $t\in (0,1]$ and $u\in [1,2]$. Then, for any $x\in D$, there exists $t_0\in [0,1]$ such that $|(x_1-ut_0,x_2-u\gamma(t_0))|\leq \epsilon$. Therefore, for any $t\in (t_0-\epsilon, t_0+\epsilon)$, there exists $\delta\in (0,\infty)$ such that $|(x_1-ut,x_2-u\gamma(t))|\leq \delta \epsilon$. Here, we used the triangle inequality and the estimate that $|\gamma(t)-\gamma(t_0)|\lesssim |t-t_0|\lesssim \epsilon$. Furthermore, we have $T(\chi_{B(0,\delta \epsilon)})(x,u)\gtrsim \epsilon \chi_{D}(x) $. As an immediate consequence of  $\|Tf\|_{L_{x,u}^q(\mathbb{R}^2\times [1,2])} \lesssim \|f\|_{L_x^{p}(\mathbb{R}^{2})}$, we obtain
       \begin{align*}
      |D|\lesssim \left|\left\{(x,u)\in\mathbb{R}^2\times[1,2]:\ T(\chi_{B(0,\delta \epsilon)})(x,u)\gtrsim \epsilon \right\}\right|\lesssim \left(\epsilon^{-1} \left\|\chi_{B(0,\delta \epsilon)}\right\|_{L_x^{p}(\mathbb{R}^{2})} \right)^{q}.
      \end{align*}
It is easy to check that $|D|\approx \epsilon$ and $\|\chi_{B(0,\delta \epsilon)}\|_{L_x^{p}(\mathbb{R}^{2})}\approx \epsilon^{\frac{2}{p}}$. This readily implies that $\epsilon^{\frac{1}{q}-\frac{2}{p}+1}\lesssim 1$ holds for all $\epsilon>0$. Thus, via letting $\epsilon\rightarrow0$ we obtain that $\frac{2}{p}-1\leq \frac{1}{q}$ and, hence, we have (ii) of Theorem \ref{thm2}.

\textbf{Proof of (iii) of Theorem \ref{thm2}}.

A simple calculation gives that the adjoint operator $T^*$ of $T$ can be written as
\begin{align*}
T^*g(x):=\int_1^2\int_0^1g(x_1+ut,x_2+u\gamma(t),u)\,\textrm{d}t \,\textrm{d}u,
\end{align*}
and the boundedness $\|Tf\|_{L_{x,u}^q(\mathbb{R}^2\times [1,2])} \lesssim \|f\|_{L_x^{p}(\mathbb{R}^{2})}$ is equivalent to
\begin{align}\label{eq:4.4}
\left\|T^*g\right\|_{L_{x}^{p'}(\mathbb{R}^2)} \lesssim \|g\|_{L_{x,u}^{q'}(\mathbb{R}^{2}\times [1,2])}.
\end{align}
For any $\epsilon>0$, let $L$ be the $\epsilon$ neighborhood of the curve $(t,\gamma(t))$ with $t\in (0,1]$. Then, for any $x\in L$, there exists $t_0\in [0,1]$ such that $|(x_1-t_0,x_2-\gamma(t_0))|\leq \epsilon$. Therefore, for any $-x\in L$, $t\in (t_0-\epsilon, t_0+\epsilon)$ and $u\in [1, 1+\epsilon]$, there exists $\delta\in (0,\infty)$ such that $|(x_1+ut,x_2+u\gamma(t))|\leq  \delta\epsilon$. Here, we used the triangle inequality and the estimates that $|\gamma(t)-\gamma(t_0)|\lesssim |t-t_0|\lesssim \epsilon$ and $|\gamma(t_0)|\lesssim 1$. As a result, by setting $W:=B(0,\delta\epsilon)\times [1,1+\epsilon]$, for $-x\in L$, we have
\begin{align*}
T^*\chi_{W}(x)\geq\int_1^{1+\epsilon}\int_{t_0-\epsilon}^{t_0+\epsilon}\chi_{W}(x_1+ut,x_2+u\gamma(t),u)\,\textrm{d}t \,\textrm{d}u\gtrsim \epsilon^2,
\end{align*}
i.e., we have $T^*\chi_{W}(x)\gtrsim \epsilon^2 \chi_{L}(-x) $. According to \eqref{eq:4.4}, we assert that
       \begin{align*}
      |L|\lesssim \left|\left\{x\in\mathbb{R}^2:\ T^*\chi_{W}(x)\gtrsim \epsilon^2 \right\}\right|\lesssim \left(\epsilon^{-2} \left\|\chi_{W}\right\|_{L_{x,u}^{q'}(\mathbb{R}^{2}\times[1,2])} \right)^{p'}.
      \end{align*}
This, combined with the fact that $|L|\approx \epsilon$ and $\|\chi_{W}\|_{L_{x,u}^{q'}(\mathbb{R}^{2}\times[1,2])}\approx \epsilon^{\frac{3}{q'}}$, implies that $\epsilon^{\frac{3}{q}-\frac{1}{p}}\lesssim 1$ holds for all $\epsilon>0$. Then, we get $\frac{1}{q}\geq \frac{1}{3p}$ by letting $\epsilon\rightarrow0$, which is (iii) of Theorem \ref{thm2}.

\textbf{Proof of (iv) of Theorem \ref{thm2}}.

It is easy to see that (i) implies (iv) of Theorem \ref{thm2}, when $\omega\in [2,\infty)$. Furthermore, from (ii) and (iii) of Theorem \ref{thm2}, the region of $(\frac{1}{p}, \frac{1}{q})$ in Theorem \ref{thm1} remains unchanged, when $\omega\in (0,\frac{3}{2})$. On the other hand, the points $D(\frac{2}{3},\frac{1}{3})$ and $C(\frac{1}{2},\frac{1}{6})$ are the endpoints in the proof. Therefore, we may conjecture that the region of $(\frac{1}{p}, \frac{1}{q})$ in Theorem \ref{thm1} also remains unchanged for all $\omega\in (0,2)$, which further implies that we can expect (iv) of Theorem \ref{thm2}. Therefore, the $\omega=2$ is a critical point, and we may need a Taylor's expansion of order $2$ as in \eqref{eq:4.5}.

We now give the proof of (iv) of Theorem \ref{thm2}. From the condition (i) of Theorem \ref{thm1}, there exists $\tau_0\in (0,1]$ such that $\gamma'(\tau_0)\neq \gamma(\tau_0)$. One may assume without loss of generality that $\tau_0=1$, $\gamma(t)$ is increasing on $(0,1]$ and $\gamma(1)=1$. Then $\gamma'(1)\neq1$. Let $e_1:=(-1, -\gamma'(1))/ \sqrt{1+\gamma'(1)^2}$ and $e_2:=(-\gamma'(1), 1)/ \sqrt{1+\gamma'(1)^2}$ be two orthogonal unit vectors. For any $\epsilon>0$, set
      $$ Q:=\left\{x\in \mathbb{R}^{2}:\ |x\cdot e_1|\leq 6 \epsilon ~~\textrm{and}~~ |x\cdot e_2|\leq [3+2|\gamma''(1)|] \epsilon^2\right\}.$$
      We choose a subset $\{u_i\}\subset [1,2]$ such that $|u_{i+1}-u_i|=2\epsilon^2$ for all $i$. For each $u_i$, denote
      $$\Omega_{u_i}:=\left\{x\in \mathbb{R}^{2}:\ |\left(x-(u_i,u_i)\right)\cdot e_1|\leq \epsilon ~\textrm{and}~ |\left(x-(u_i,u_i)\right)\cdot e_2|\leq \epsilon^2 \right\}.$$
      For any $x\in \Omega_{u_i}$ and $t\in[1-\epsilon/\sqrt{1+\gamma'(1)^2}, 1]$, Taylor's expansion yields
      \begin{align}\label{eq:4.5}
      \gamma(t)=1+\gamma'(1)(t-1)+\frac{\gamma''(1)}{2}(t-1)^2+o\left((t-1)^2\right),
      \end{align}
      where $o((t-1)^2)$ means $o((t-1)^2)/(t-1)^2\rightarrow 0$ as $t\rightarrow 1^-$. Furthermore, by \eqref{eq:4.5}, a simple calculation gives $|((u_it,u_i\gamma(t))-(u_i,u_i))\cdot e_1|\leq 4 \epsilon$ and $|((u_it,u_i\gamma(t))-(u_i,u_i))\cdot e_2|\leq 2|\gamma''(1)|\epsilon^2$ for $\epsilon>0$ small enough. By the triangle inequality, we may conclude that $|(x-(u_it,u_i\gamma(t)))\cdot e_1|\leq 5\epsilon$ and $|(x-(u_it,u_i\gamma(t)))\cdot e_2|\leq [1+2|\gamma''(1)|] \epsilon^2$. Therefore, for any $u\in (u_i-\epsilon^2, u_i+\epsilon^2)$, the triangle inequality also gives $|(x-(ut,u\gamma(t)))\cdot e_1|\leq 6\epsilon$ and $|(x-(ut,u\gamma(t)))\cdot e_2|\leq [3+2|\gamma''(1)|] \epsilon^2$, where $\epsilon>0$ small enough.

      The above argument implies that $(x_1-ut, x_2-u\gamma(t))\in Q$ if $x\in \Omega_{u_i}$, $u\in (u_i-\epsilon^2, u_i+\epsilon^2)$ and $t\in[1-\epsilon/\sqrt{1+\gamma'(1)^2}, 1]$. Consequently, for $\epsilon>0$ small enough, one may obtain
      \begin{align*}
     T\chi_Q(x,u)\geq \int_{1-\frac{\epsilon}{\sqrt{1+\gamma'(1)^2}}}^1\chi_{Q}(x_1-ut, x_2-u\gamma(t))\,\textrm{d}t= \frac{\epsilon}{\sqrt{1+\gamma'(1)^2}}
      \end{align*}
      for any $x\in \Omega_{u_i}$ and $u\in (u_i-\epsilon^2, u_i+\epsilon^2)$. By $\gamma'(1)\neq1$, we have $\Omega_{u_i}\cap \Omega_{u_{i'}}=\emptyset$ for any $i\neq i'$. Notice that the number of the elements in $\{u_i\}\subset [1,2]$ is equivalent to $\epsilon^{-2}$. Then clearly
       \begin{align*}
     \left\|T\chi_Q\right\|_{L_{x,u}^{q}(\mathbb{R}^{2}\times [1,2])} \geq \left(\sum_i \int_{u_i-\epsilon^2}^{u_i+\epsilon^2}\int_{\Omega_{u_i}} \left|T\chi_Q\right|^q \,\textrm{d}x\,\textrm{d}u\right)^{\frac{1}{q}}\gtrsim \epsilon^{1+\frac{3}{q}}.
      \end{align*}
This, combined with $\|\chi_Q\|_{L^{p}(\mathbb{R}^{2})}\approx \epsilon^{\frac{3}{p}}$ and $\|Tf\|_{L_{x,u}^q(\mathbb{R}^2\times [1,2])} \lesssim \|f\|_{L_x^{p}(\mathbb{R}^{2})}$, leads to that $\epsilon^{1+\frac{3}{q}-\frac{3}{p}}\lesssim 1$ holds for all $\epsilon>0$ small enough. Let $\epsilon\rightarrow0$, it follows that $\frac{1}{q}\geq\frac{1}{p}-\frac{1}{3}$. Thus (iv) of Theorem \ref{thm2} is proved.

Therefore, putting things together we finish the proof of Theorem \ref{thm2}.

\bigskip

\noindent  Junfeng Li

\smallskip

\noindent  School of Mathematical Sciences, Dalian University of Technology, Dalian, 116024,  People's Republic of China

\smallskip

\noindent {\it E-mail}: \texttt{junfengli@dlut.edu.cn}

\bigskip

\noindent  Naijia Liu

\smallskip

\noindent  School of Mathematics, Sun Yat-sen University, Guangzhou, 510275,  People's Republic of China

\smallskip

\noindent {\it E-mail}: \texttt{liunj@mail2.sysu.edu.cn}

\bigskip

\noindent Zengjian Lou and Haixia Yu (Corresponding author)

\smallskip

\noindent  Department of Mathematics, Shantou University, Shantou, 515821, People's Republic of China

\smallskip

\noindent {\it E-mails}: \texttt{zjlou@stu.edu.cn} (Z. Lou)

\noindent\phantom{{\it E-mails:}} \texttt{hxyu@stu.edu.cn} (H. Yu)


\begin{thebibliography}{10}

\bibitem{B} D. Beltran, J. Hickman and C.D. Sogge, Variable coefficient Wolff-type inequalities and sharp local smoothing estimates for wave equations on manifolds, Anal. PDE 13 (2020), no. 2, 403--433.

\vspace{-.3cm}

\bibitem{Bour86} J. Bourgain, Averages in the plane over convex curves and maximal operators, J. Analyse Math. 47 (1986), 69--85.

\vspace{-.3cm}

\bibitem{BoD} J. Bourgain and C. Demeter, The proof of the $l^2$ decoupling conjecture, Ann. of Math. 182 (2015), no. 1, 351--389.

\vspace{-.3cm}

\bibitem{CCVWW} A. Carbery, M. Christ, J. Vance, S. Wainger and D. Watson, Operators associated to flat plane curves: $L^p$ estimates via dilation methods, Duke Math. J. 59 (1989), no. 3, 675--700.

\vspace{-.3cm}

\bibitem{CVWW} A. Carbery, J. Vance, S. Wainger and D. Watson, The Hilbert transform and maximal function along flat curves, dilations, and differential equations, Amer. J. Math. 116 (1994), no. 5, 1203--1239.

\vspace{-.3cm}

\bibitem{Cho} Y. Choi, The  $L^p-L^q$  mapping properties of convolution operators with the affine arclength measure on space curves, J. Aust. Math. Soc. 75 (2003), no. 2, 247--261.

\vspace{-.3cm}

\bibitem{Chri} M. Christ, Convolution, curvature, and combinatorics: a case study, Int. Math. Res. Not. IMRN (1998), no. 19, 1033--1048.

\vspace{-.3cm}

\bibitem{DH72} J.J. Duistermaat and L. $\textrm{H}\ddot{\textrm{o}}\textrm{rmander}$, Fourier integral operators. II,
Acta Math. 128 (1972), no. 3--4, 183--269.

\vspace{-.3cm}

\bibitem{GLMX} C. Gao, B. Liu, C. Miao and Y. Xi, Square function estimates and local smoothing for Fourier integral operators, Proc. Lond. Math. Soc. 126 (2023), no. 6, 1923--1960.

\vspace{-.3cm}

\bibitem{Gr1} P.T. Gressman, $L^p$-improving properties of $X$-ray like transforms, Math. Res. Lett. 13 (2006), no. 5--6, 787--803.

\vspace{-.3cm}

\bibitem{Gr2} P.T. Gressman, Uniform sublevel Radon-like inequalities, J. Geom. Anal. 23 (2013), no. 2, 611--652.

\vspace{-.3cm}

\bibitem{GWZ} L. Guth, H. Wang and R. Zhang, A sharp square function estimate for the cone in $\mathbb{R}^3$, Ann. of Math. 192 (2020), no. 2, 551--581.

\vspace{-.3cm}

\bibitem{H} J. Hickman, Uniform  $L^p_x-L^q_{x,r}$ improving for dilated averages over polynomial curves, J. Funct. Anal. 270 (2016), no. 2, 560--608.

\vspace{-.3cm}

\bibitem{Ho} L. $\textrm{H}\ddot{\textrm{o}}\textrm{rmander}$, Estimates for translation invariant operators in $L^p$  spaces, Acta Math. 104 (1960), 93--140.

\vspace{-.3cm}

\bibitem{Ho71} L. $\textrm{H}\ddot{\textrm{o}}\textrm{rmander}$, Fourier integral operators. I, Acta Math. 127 (1971), no. 1--2, 79--183.

\vspace{-.3cm}

\bibitem{IKM} I.A. Ikromov, M. Kempe and D.  M\"{u}ller, Estimates for maximal functions associated with hypersurfaces in ${\mathbb R}^{3}$ and related problems of harmonic analysis, Acta Math. 204 (2010), no. 2, 151--271.

\vspace{-.3cm}

\bibitem{Iose} A. Iosevich, Maximal operators associated to families of flat curves in the plane, Duke Math. J. 76 (1994), no. 2, 633--644.

\vspace{-.3cm}

\bibitem{KLO} H. Ko, S. Lee and S. Oh, Maximal estimates for averages over space curves, Invent. Math. 228 (2022), no. 2, 991--1035.

\vspace{-.3cm}

\bibitem{LeeL} J. Lee and S. Lee, $L^p-L^q$ estimates for the circular maximal operator on Heisenberg radial functions, Math. Ann. 385 (2023), no. 3--4, 1521--1544.

\vspace{-.3cm}

\bibitem{LWZ} W. Li, H. Wang and Y. Zhai, $L^p$-improving bounds and weighted estimates for maximal functions associated with curvature, J. Fourier Anal. Appl. 29 (2023), no. 1, Paper No. 10, 63 pp.

\vspace{-.3cm}

\bibitem{Lit} W. Littman, $L^p-L^q$-estimates for singular integral operators arising from hyperbolic equations,
Proc. Sympos. Pure Math., Vol. XXIII. American Mathematical Society, Providence, RI, 1973, pp. 479--481.


\vspace{-.3cm}

\bibitem{LSY} N. Liu, L. Song and H. Yu, $L^p$ bounds of maximal operators along variable planar curves in the Lipschitz regularity, J. Funct. Anal. 280 (2021), no. 5, Paper No. 108888, 40 pp.

\vspace{-.3cm}

\bibitem{LYu} N. Liu and H. Yu, Hilbert transforms along variable planar curves: Lipschitz regularity, J. Funct. Anal. 282 (2022), no. 4, Paper No. 109340, 36 pp.

\vspace{-.3cm}

\bibitem{LiuYu} N. Liu and H. Yu, $L^p$-improving bounds of maximal functions along planar curves, arXiv:2309.01992.

\vspace{-.3cm}

\bibitem{MSS} G. Mockenhaupt, A. Seeger and C.D. Sogge, Local smoothing of Fourier integral operators and Carleson-Sj\"olin estimates, J. Amer. Math. Soc. 6 (1993), no. 1, 65--130.

\vspace{-.3cm}

\bibitem{MSS92} G. Mockenhaupt, A. Seeger and C.D. Sogge, Wave front sets, local smoothing and Bourgain's circular maximal theorem, Ann. of Math. 136 (1992), no. 1, 207--218.

\vspace{-.3cm}

\bibitem{NRW76} A. Nagel, N. Riviere and S. Wainger, A maximal function associated to the curve $(t,t^2)$, Proc. Natl. Acad. Sci. USA. 73 (1976), no. 3, 1416--1417.

\vspace{-.3cm}

\bibitem{Ob1} D.M. Oberlin, Convolution with measures on flat curves in low dimensions, J. Funct. Anal. 259 (2010), no. 7, 1799--1815.

\vspace{-.3cm}

\bibitem{Ob2} D.M. Oberlin, Some convolution inequalities and their applications, Trans. Amer. Math. Soc. 354 (2002), no. 6, 2541--2556.

\vspace{-.3cm}

\bibitem{Sch} W. Schlag, A generalization of Bourgain's circular maximal theorem, J. Amer. Math. Soc. 10 (1997), no. 1, 103--122.

\vspace{-.3cm}

\bibitem{SchS} W. Schlag and C.D. Sogge, Local smoothing estimates related to the circular maximal theorem, Math. Res. Lett. 4 (1997), no. 1, 1--15.

\vspace{-.3cm}

\bibitem{SSS} A. Seeger, C.D. Sogge and E.M. Stein, Regularity properties of Fourier integral operators, Ann. of Math. 134 (1991), no. 2, 231--251.

\vspace{-.3cm}

\bibitem{So} C.D. Sogge, Fourier integrals in classical analysis. Second edition. Cambridge Tracts in Mathematics, 210. Cambridge University Press, Cambridge, 2017.

\vspace{-.3cm}

\bibitem{Sog} C.D. Sogge, Propagation of singularities and maximal functions in the plane, Invent. Math. 104 (1991), no. 2, 349--376.

\vspace{-.3cm}

\bibitem{SSte} C.D. Sogge and E.M. Stein, Averages of functions over hypersurfaces in $\mathbb{R}^{n}$, Invent. Math. 82 (1985), no. 3, 543--556.

\vspace{-.3cm}

\bibitem{Stein} E.M. Stein, Harmonic Analysis: real-variable methods, orthogonality, and oscillatory integrals, Princeton Mathematical Series 43, Monographs in Harmonic Analysis III, Princeton University Press, Princeton, NJ, 1993.

\vspace{-.3cm}

\bibitem{Ste1} E.M. Stein, Maximal functions. II. Homogeneous curves, Proc. Nat. Acad. Sci. U.S.A. 73 (1976), no. 7, 2176--2177.

\vspace{-.3cm}

\bibitem{SW1} E.M. Stein and S. Wainger, Maximal functions associated to smooth curves, Proc. Nat. Acad. Sci. U.S.A. 73 (1976), no. 12, 4295--4296.

\vspace{-.3cm}

\bibitem{SW} E.M. Stein and S. Wainger, Problems in harmonic analysis related to curvature, Bull. Amer. Math. Soc. 84 (1978), no. 6, 1239--1295.

\vspace{-.3cm}

\bibitem{Sto1} B. Stovall, Endpoint bounds for a generalized Radon transform, J. Lond. Math. Soc. 80 (2009), no. 2, 357--374.

\vspace{-.3cm}

\bibitem{Sto2} B. Stovall, Endpoint $L^p\rightarrow L^q$ bounds for integration along certain polynomial curves,
J. Funct. Anal. 259 (2010), no. 12, 3205--3229.

\vspace{-.3cm}

\bibitem{Str} R.S. Strichartz, Restrictions of Fourier transforms to quadratic surfaces and decay of solutions of wave equations, Duke Math. J. 44 (1977), no. 3, 705--714.

\vspace{-.3cm}

\bibitem{TW} T. Tao and J. Wright, $L^p$ improving bounds for averages along curves, J. Amer. Math. Soc. 16 (2003), no. 3, 605--638.

\end{thebibliography}
\end{document}